\documentclass[12pt,letterpaper]{amsart}
\usepackage{amsmath, amsthm, amssymb, mathrsfs}
\usepackage[tmargin=1.35in,bmargin=1.2in,rmargin=1.4in,lmargin=1.4in]{geometry}
\usepackage[breaklinks=true]{hyperref}
\usepackage{tikz-cd}

\theoremstyle{plain}
\newtheorem{theorem}{Theorem}[section]

\newtheorem{assumption}{Assumption}
\newtheorem{lemma}{Lemma}[section]
\newtheorem{prop}{Proposition}[section]
\newtheorem{cor}{Corollary}[section]

\theoremstyle{definition}
\newtheorem{defin}{Definition}[section]
\newtheorem{remark}{Remark}[section]
\newtheorem*{acknowledgement}{Acknowledgements}

\newcommand{\mf}[1]{\displaystyle{\mathfrak{#1}}}

\newcommand{\comment}[1]{}

\DeclareMathOperator{\spec}{\ensuremath{Spec}}

\DeclareMathOperator{\Gr}{\ensuremath{gr}}

\DeclareMathOperator{\ad}{\ensuremath{ad}}
\DeclareMathOperator{\Id}{\ensuremath{Id}}

\DeclareMathOperator{\Sym}{\ensuremath{Sym}}

\DeclareMathOperator{\End}{\ensuremath{End}}

\DeclareMathOperator{\rank}{\ensuremath{rank}}
\DeclareMathOperator{\Aut}{\ensuremath{Aut}}
\DeclareMathOperator{\Pic}{\ensuremath{Pic}}
\DeclareMathOperator{\Ad}{\ensuremath{Ad}}
\DeclareMathOperator{\Out}{\ensuremath{Out}}

\begin{document}
\title{Rigidity  of quantum algebras}
\author{Akaki Tikaradze}
\email{Akaki.Tikaradze@utoledo.edu}
\address{University of Toledo, Department of Mathematics \& Statistics, 
Toledo, OH 43606, USA}


\begin{abstract}

Given an associative $\mathbb{C}$-algebra $A$, we call $A$ strongly
rigid if for any pair of
finite subgroups of its automorphism groups $G, H,$ such that
$A^G\cong A^H$, then $G$ and $H$ must be isomorphic. In this paper, we show that a large class of filtered quantizations
are strongly rigid.
We also solve the inverse Galois problem
for a wide class of rational Cherednik algebras that includes all (simple) classical generalized Weyl algebras, and also for quantum tori.
 Finally, we show that the Picard group of an $n$-dimensional quantum torus is isomorphic to the group of its outer automorphisms.

\end{abstract}

\maketitle

\section{Introduction}

It was shown by Alev and Polo \cite{AP} that if an associative $\mathbb{C}$-algebra $A$ is either an enveloping
algebra of a semi-simple Lie algebra or the $n$-th Weyl algebra, then for any nontrivial finite subgroup $\Gamma$
of automorphisms of $A$,  the fixed ring $A^{\Gamma}$ is not isomorphic to $A.$ Such a property of an algebra
is referred to as its rigidity.
On the other hand, it was proved by Alev, Hodges and Velves \cite{AHV} that given a pair of finite subgroups of automorphisms $G, H$ of the first Weyl algebra $A_1(\mathbb{C}),$
such that corresponding fixed rings are isomorphic $A_1(\mathbb{C})^G\cong A_1(\mathbb{C})^H,$ then $G\cong H.$
In \cite{T3} we generalized the above mentioned results by showing that
if $W, W'$ are finite subgroups of automorphisms of $A$ ($A$ is still either the Weyl algebra or an enveloping algebra
of a semi-simple Lie algebra) such that $A^W\cong A^{W'}$ (in fact, it suffices to assume that $A^W$ and $A^{W'}$ are derived equivalent), then $W\cong W'.$

 It will be convenient to use the following terminology.
 
 \begin{defin}
 
 An associative $\mathbb{C}$-algebra $A$ is said to be strongly rigid, if given a pair $W, W'$ of finite susbgroups of
 $\mathbb{C}$-algebra automorphisms of $A$ such that $A^W\cong A^{W'},$ then $W\cong W'.$ A $\mathbb{C}$-algebra $A$
 is rigid if for any nontrivial finite subgroup $W\leq \Aut_{\mathbb{C}}(A),$ we have $A\neq A^W.$
 \end{defin}

It follows easily from our results in \cite{T4} that if $X$ is a smooth affine algebraic variety over $\mathbb{C}$ whose algebraic fundamental group is finite (respectively trivial), then
$D(X)$ is rigid (respectively strongly rigid).

Since our main tool for studying rigidity questions of quantum algebras is reduction modulo a large prime,
we consider those quantizations for which the center $\mod p$ and corresponding Poisson bracket are easily described, for at least infinitely many primes.
Before giving a definition of such quantizations, we first recall the definition of the reduction $\mod p$ Poisson bracket.

It is well-known that given an associative flat $\mathbb{Z}$-algebra $R$ and a
prime number $p,$ the center $Z(R/pR)$ of its reduction mod $p$ acquires
a natural Poisson bracket (see for example \cite[Section 5.2]{BKK}), which
we refer to as the reduction $\mod p$ bracket, defined as follows. Given
$a, b\in Z(R/pR)$, let $z, w\in R$ be their respective lifts. 
Then the reduction $\mod p$ Poisson bracket $\lbrace a, b \rbrace$ is defined to be

$$\frac{1}{p}[z, w] \mod p\in Z(R/pR).$$

Let $S$ be a commutative ring, let $\mathcal{O}$ be a nonnegatively graded affine Poisson $S$-algebra of a negative degree Poisson bracket. Let $A$ be an $S$-algebra equipped with
an ascending $S$-algebra filtration such that $\Gr(A)=\mathcal{O}.$ Then we say that $A$ is a filtered quantization of $\mathcal{O}$ if
$$[x, y]=\lbrace \bar{x}, \bar{y}\rbrace+\text{low degree terms}, \quad x, y\in A,$$
where $\bar{x}, \bar{y}$ denote the top symbols of $x, y$ in $\Gr(A).$

The next definition is predicated on the following fact from commutative algebra, which is a direct consequence of the Chebotarev density
theorem as shown in [\cite{VWW}, Theorem 1.1].

\begin{theorem}\label{weak-Chebotarev}
Let $S$ be a finitely generated integral domain. Then for infinitely many primes $p$, there exists
a ring homomorphism $S\to \mathbb{F}_p.$
\end{theorem}

\begin{defin}
Let $S\subset\mathbb{C}$ be a finitely generated subring.
Let $\mathcal{O}=\bigoplus_{n\geq 0}\mathcal{O}_n$ be a nonnegatively graded affine Poisson $S$-domain, such
that $\mathcal{O}_n$ is a finite rank free $S=\mathcal{O}_0$-module, $n\geq 0.$ Let an $S$-algebra $A$ be a filtered quantization of $\mathcal{O}$. We say that $A$ is a good quantization of
$\mathcal{O},$ equivalently of a conical affine Poisson variety $\spec(\mathcal{O})$, if
there exists a localization $S'=S_f, f\neq 0$ such that the following holds. For any prime $p$ and a ring homomorphism $\chi:S'\to\mathbb{F}_p$ there exists a 
subalgebra $Z_p\subset Z(A_{\mathbb{F}_p})$ 
and an isomorphism of algebras $$\iota:\mathcal{O}_{\mathbb{F}_p}^p\to Z_p,$$ 
 such that $\Gr(\iota)=\Id$ and $\iota$ interchanges
the reduction $\mod p$ Poisson bracket on $Z(A_{\mathbb{F}_p})$ with the minus of the Poisson bracket on $\mathcal{O}_{\mathbb{F}_p}^p.$

If $A$ is a filtered quantization of an affine Poisson $\mathbb{C}$-domain $\mathcal{O}$, then we call
$A$ a good quantization if there exists a finitely generated subring $S\subset\mathbb{C}$ and models $A_S, \mathcal{O}_S$ of $A, \mathcal{O}$ respectively, such that $A_S$ is a good quantization of $\mathcal{O}_S$ over $S.$
\end{defin}

Typical examples of good quantizations are enveloping algebras of algebraic Lie algebras and rings of differential operators on smooth varieties over
$\mathbb{C}.$
More generally, algebras arising from (certain) quantum Hamiltonian reductions (see Proposition \ref{good-Qow}).

The following is one of the main results of the paper. It is a significant strengthening of the previously mentioned rigidity results in the literature.

\begin{theorem}\label{rigidity}

Let $X$ be an affine normal conical Poisson variety over $\mathbb{C}$, whose smooth locus is symplectic with a finite algebraic fundamental group.
Let $A$ be a good filtered quantization of $X.$ Then $A$ is rigid. If in addition the smooth locus of $X$ is simply connected, then
$A$ is strongly rigid.

\end{theorem}

The following result shows nonexistence of injective homomorphisms between enveloping algebras
 of semi-simple Lie algebras of the same dimension.

\begin{theorem}\label{enveloping}

Let $\mathfrak{g}, \mathfrak{g'}$ be a pair of non-isomorphic complex semi-simple Lie algebras of equal dimension. 
Then there are no injective $\mathbb{C}$-algebra homomorphisms between $U(\mathfrak{g})$ and $U(\mathfrak{g}').$ Assume
in addition that $\rank(\mathfrak{g})=\rank(\mathfrak{g'}).$
 Let $\chi, \chi'$ be central characters of $U(\mathfrak{g}), U(\mathfrak{g}').$ 
 Then there are no injective $\mathbb{C}$-algebra homomorphisms between $U_{\chi}(\mathfrak{g})$ and $U_{\chi'}(\mathfrak{g'}).$

\end{theorem}

We also have the following general result that provides  a useful upper bound on finite groups of automorphisms
of filtered quantizations. To state it, we need to recall that given a Poisson algebra $B$ with a maximal Poisson ideal $m$
and the residue field $\bold{k}=B/m,$ then the Poisson bracket defines a $\bold{k}$-Lie algebra
structure on $m/m^2.$

\begin{theorem}\label{automorphisms}

Let $B$ be a Poisson $\mathbb{C}$-domain with the unique Poisson maximal ideal $m$. Let
$A$ be a good filtered quantization of $B.$ Then any finite subgroup of automorphisms
of $A$ is isomorphic to a subgroup of Lie algebra automorphisms of $m/m^2.$

\end{theorem}

This result is applicable to spherical subalgebra of symplectic reflection algebras and central reductions of
finite $W$-algebras. More generally, it is well-suitable for algebras obtained via quantum Hamiltonian reduction.
The next result is a generalization of our earlier result on finite subgroups of automorphisms
of enveloping algebras (\cite{T2}) to finite $W$-algebras.

\begin{cor}\label{W-algebra}
Let $e\in\mathfrak{g}$ be a nilpotent element of a complex  semi-simple Lie algebra, and $\chi:Z(U(\mathfrak{g}))\to\mathbb{C}$
be a central character. Let $W(e, \chi)$ denote the central reduction of the finite $W$-algebra of $e$ with respect to the character $\chi.$
Then any finite subgroup of automorphisms of $W(e, \chi)$ is isomorphic to a subgroup of Lie algebra automorphisms
of $\mathfrak{g}(e).$

\end{cor}

Given a simple Noetherian $\mathbb{C}$-domain $A$, by the inverse Galois problem for $A$ we understand
classifying all finite groups $G$ (up to isomorphisms) for which there exists a $\mathbb{C}$-domain $B$ equipped with a faithful action of $G$
by $\mathbb{C}$-algebra automorphisms, such that $A\cong B^{G}.$  We solve the inverse Galois problem
for several classes of spherical subalgebras of rational Cherednik algebras which include all (classical) generalized Weyl algebras, as well as quantum tori.

Given an associative ring $R$, recall that its Picard group, denoted by $\Pic(R)$, is defined as the group of isomorphism classes
of invertible  $R$-bimodules under the tensor product. This group is closely related to the automorphism group of $R.$ Namely, there is a natural
homomorphism $\Aut(R)\to \Pic(R)$, whose kernel consists of inner automorphisms of $R.$ It is known that this map is surjective for many
important classes of quantum algebras such as: the first Weyl algebra by classical works of Dixmier and Stafford \cite{St}, and quantum tori by a result of
Berest, Ramadoss and Tang \cite{BRT}. In this paper we describe the Picard group of the $n$-quantum tori (tensor product of quantum tori) 
Theorem \ref{picard}. Our
proof yet again relies on the reduction to a large prime technique, thus it is\ similar in spirit (and motivated by) to the one by Stafford.

\section{Examples of good filtered quantizations}

In this section, we explain how the enveloping algebras of algebraic Lie algebras and a large class of algebras obtained from quantum Hamiltonian reductions (which includes finite W-algebras) give rise to good filtered quantizations.

 First, we recall a key computation of the reduction $\mod p$ Poisson bracket for restricted Lie
algebras due to Kac and Radul \cite{KR}. 
Let $R$ be a  commutative reduced ring of  characteristic $p>0.$
Let $\mathfrak{g}$ be a restricted Lie algebra over $R$
with the restricted structure map $g\to g^{[p]}, g\in \mathfrak{g}.$
As usual,  $Z_p(\mathfrak{g})$ denotes the  $p$-center of $U(\mathfrak{g})$: the central $R$-subalgebra of the enveloping algebra $U(\mathfrak{g})$
generated by elements of the form $g^p-g^{[p]}, g\in \mathfrak{g}.$ It is well-known that the map $g\to g^{p}-g^{[p]}$ induces
homomorphism of $R$-algebras $$i:\Sym (\mathfrak{g})\to Z_p(\mathfrak{g}),$$
where $Z_p(\mathfrak{g})$ is viewed as an $R$-algebra via the Frobenius map $F:R\to R.$
The homomorphism $i$ is an isomorphism when $R$ is perfect.
Recall also that
the Lie algebra bracket on $\mathfrak{g}$ defines the Kirillov-Kostant Poisson bracket on the symmetric algebra $\Sym(\mathfrak{g}).$


The following is the above mentioned key result from \cite{KR}. Throughout the paper, given an abelian group $V$, by $V_p$ we denote $V/pV.$

\begin{lemma}\label{kac}
Let $S$ be a finitely generated integral domain over $\mathbb{Z}.$
Let $\mathfrak{g}$ be an algebraic Lie algebra over $S.$
Then $Z_p(\mathfrak{g}_p)$ is a Poisson subalgebra of $Z(U(\mathfrak{g}_p)),$ where $Z(U(\mathfrak{g}_p))$ is
equipped with the reduction $\mod p$ Poisson bracket.  Moreover, the induced Poisson
bracket on $Z_p(\mathfrak{g}_p)$ coincides with the negative of the Kirrilov-Kostant bracket:
$$\lbrace a^p-a^{[p]}, b^p-b^{[p]}\rbrace=-([a, b]^p-[a,b]^{[p]}),\quad a\in \mathfrak{g}_p, b\in \mathfrak{g}_p.$$

\end{lemma}
Now Lemma \ref{kac} immediately implies that if $\mathfrak{g}$ is a Lie algebra of an alegbraic group over $\mathbb{C}$,
then the enveloping algebra $U(\mathfrak{g})$ is a good quantization of 
$\Sym(\mathfrak{g})$ (equipped with the Kirillov-Kostant Poisson bracket). A similar computation shows that if $X$ is a smooth affine variety over
$\mathbb{C},$ then the algebra of differential operators $D(X)$ is a good quantization of the cotangent bundle $T^{*}(X).$

Let a reductive algebraic group $G$ act on a smooth affine algebraic variety $X$ over $\mathbb{C}.$
Let $\mathfrak{g}$ be the Lie algebra of $G.$
Let $\mu:T^*(X)\to \mathfrak{g}^*$ be the corresponding moment map.
We will assume that this map is flat, and for generic $G$-invariant character $\chi\in \mathfrak{g}^*$ the action of $G$
on $\mu^{-1}(\chi)$ is free. 
   
       Given a $G$-invariant character $\chi\in \mathfrak{g}^*$, denote by $U_{\chi}(G, X)$ the quantum Hamiltonian reduction of
       $D(X)$ with respect to $\chi$. So, 
       $$U_{\chi}(G, X)=(D(X)/D(X)\mathfrak{g}^{\chi})^G,$$ where 
       $$\mathfrak{g}^{\chi}=\lbrace g-\chi(g)\in D(X), \quad g\in\mathfrak{g}\rbrace.$$
The usual filtration on $D(X)$ by the order of differential operators induces the corresponding filtration on $U_{\chi}(G, X).$
Then it follows from the flatness of the moment map that 
$$\Gr(U_{\chi}(G, X))=\mathcal{O}(\mu^{-1}(0)//G).$$

Next,  we need to recall some results and notations associated with quantum Hamiltonian reduction of the ring of crystalline
differential operators in characteristic $p>0$ from \cite{BFG}. 

Let $X$ be a smooth affine variety over an algebraically closed field $\bold{k}$ of characteristic $p$, 
and $G$ be a reductive algebraic group over $\bold{k}$ with the Lie algebra $\mf{g}.$ Let $G$ act on $X.$
Denote by $D(X)$  the ring of crystalline differential operators on $X.$
As before, we have the moment map $\mu:T^*(X)\to\mathfrak{g}^*$ and the algebra homomorphism  $U(\mf{g})\to D(X).$
Recall that we have the canonical isomorphism 
$$i:\text{Sym}(\mf{g})^{(1)}\to Z_p(\mathfrak{g}).$$
 On the other hand,  the center of $D(X)$ 
is generated over $\mathcal{O}(X)^p$ by $$\xi^p-\xi^{[p]}, \xi\in T_X.$$ This leads to the canonical isomorphism
$$\mathcal{O}(T^*(X))^{(1)}\to Z(D(X)).$$
 We have the homomorphism $\eta':Z_p(\mf{g})\to Z(D(X))$ and the corresponding homomorphism
$$\eta:\text{Sym}(\mf{g})^{(1)}\to \mathcal{O}(T^*(X))^{(1)}.$$
Given $\chi\in \mathfrak{g}^*$, then $\chi^{[1]}\in \mathfrak{g}^*$ is defined as follows: 
$$\chi^{[1]}(g)=\chi(g)^p-\chi(g^{[p]}),\quad g\in \mathfrak{g}.$$
Using the above homomorphisms $\eta, \eta',$ it follows that the center of $U_{\chi}(G, X)$ contains 
a subring $Z_0$ isomorphic to the Frobenius twist of $\mathcal{O}(\mu^{-1}(\chi^{[1]})//G).$
In this setting the following holds.

\begin{lemma}\label{Azumaya}\cite{BFG}
Let $\chi\in (\mathfrak{g}^*)^G$ be a character.  
Then $U_{\chi}(G, X)$ is a finite algebra over  $\mu^{-1}(\chi^{[1]})//G.$ If $G$ acts freely of $\mu^{-1}(\chi^{[1]}),$
then $U_{\chi}(G, X)$ is an Azumaya algebra over $\mu^{-1}(\chi^{[1]})//G.$

\end{lemma}

Now going back to the characteristic 0 setting, we have the following.

\begin{prop}\label{good-Qow}
Let $G$ be a complex reductive algebraic group acting (with the Lie algebra $\mathfrak{g})$ on a smooth affine variety $X$ over $\mathbb{C}.$
Assume that the moment map $\mu:T^*(X)\to \mathfrak{g}^*$ is flat and  $\mu^{-1}(0)//G$ is a normal variety with an open symplectic subset.
Then for any character $\chi$ the corresponding quantum Hamiltonian reduction $U_{\chi}(G, X)$ is a good quantization of $\mu^{-1}(0)//G.$

\end{prop}

\begin{proof}
Let $S \subset\mathbb{C}$ be a finitely generated subring containing values of $\chi$ over which everything is defined.
Then given a base change $S\to\bold{k}$ to an algebraically closed field of sufficiently large characteristic,
we have $$Z(U_{\chi}(G_{\bold{k}}, X_{\bold{k}}))=\mathcal{O}(\mu^{-1}(\chi^{[1]})//G_{\bold{k}}).$$
Given any $\rho: S\to \mathbb{F}_p$, then clearly $\rho(\chi)^{[1]}=0.$
Thus $$Z_0\cong \mathcal{O}(\mu^{-1}(\chi^{[1]})//G_{\mathbb{F}_p})^p=\mathcal{O}(\mu^{-1}(0)//G_{\mathbb{F}_p}),$$ and $\Gr(Z_0)=\mathcal{O}(\mu^{-1}(0)//G_{\mathbb{F}_p})^p.$
So, $U_{\chi}(G, X)$ is a good quantization of $\mu^{-1}(0)//G.$
\end{proof}

Let $\mathfrak{g}$ be a semi-simple Lie algebra, $e\in \mathfrak{g}$ a nilpotent element, $\chi: Z(U\mathfrak{g})\to\mathbb{C}$
be a central character. Denote by $S_{e}$ the Slodowy slice at $e$, and by $\mathcal{N}$ the nilpotent cone. Recall that $W_{\chi}(\mathfrak{g}, e)$
is equipped with the Kazhdan filtration so that
 $$\Gr(W_{\chi}(\mathfrak{g}, e))=\mathcal{O}(S_{e}\cap \mathcal{N}).$$
Then we have the following result whose proof is very similar to that of Proposition \ref{good-Qow}.
\begin{prop}\label{good-W}
With the above notation, the
algebra $W_{\chi}(\mathfrak{g}, e)$ is a good filtered quantization of $S_{e}\cap \mathcal{N}.$

\end{prop}

\begin{proof}

We will check that $U_{\chi}(\mathfrak{g})$ is a good quantization of $\mathcal{N}$, which easily implies the general case.
Let $G$ be the corresponding semi-simple algebraic group. Let $S$ be a finitely generated ring containing values of $\chi.$
Then for large enough $p,$ given any base change $S\to \mathbb{F}_p,$ then $\bar{\chi}$ corresponds to a central character
of a finite dimensional representation of $G_{\mathbb{F}_p}.$ Then it is well-known that 
$Z(U_{\bar{\chi}}(\mathfrak{g}_{\mathbb{F}_p}))$ can be identified with $\mathcal{O}(\mathcal{N}_{\mathbb{F}_p})^p.$ 
Thus $U_{\chi}(\mathfrak{g})$ is a good filtered quantization of $\mathcal{N}.$
\end{proof}

\section{Rigidity results}

The results of this section (in fact, the whole paper) are motivated by fundamental papers
of Belov-Kanel and Kontsevich \cite{BKK} and Tsuchimoto \cite{Ts} on automorphisms of Weyl algebras.
By considering reduction of the Weyl algebra $\mod p\gg 0$, they defined and studied a canonical homomorphism
$$\Aut(A_n(\mathbb{C}))\to S\Aut(\mathbb{A}^{2n}),$$
where $S\Aut(\mathbb{A}^{2n})$ denotes the group of symplectomorphisms of $\mathbb{A}^{2n}$ equipped with the standard symplectic form.
This homomorphism was conjectured to be an isomorphism.

The goal of this section is to introduce and study a ``dequantization'' functor $Z_{\infty}$ from a category of
 good quantizations to the category of Poisson algebras. In particular, we generalize
the above homomorphism of Belov-Kanel and Kontsevich.

More specifically, given a homomorphism of $S$-domains (where $S\subset\mathbb{C}$ is a finitely generated ring) $\phi:A\to B$ we would like to know when
the corresponding reduction modulo $p$ (for large enough primes $p$) homomorphism $\phi_{p}:A_p\to B_p$ preserves the center: 
$$\phi_p(Z(A_p))\subset Z(B_p).$$
If this is the case, we get a nice functor $Z_{\infty}$ from the category of certain $S$-algebras to $S_{\infty}$-Poisson algebras, where 
$S_{\infty}$ denotes the reduction of $S$ modulo the infinite prime
$$S_{\infty}=(\prod_pS/pS)/\bigoplus_p S/pS.$$
Recall that $\mathbb{C}_{\infty}$ is defined as the direct limit of $S_{\infty}$ over all finitely generated subrings $S\subset\mathbb{C}.$
We also consider the following variant of $S_{\infty}$ which is convenient while dealing with good quantizations. Given a finitely generated commutative domain
$S,$ we put
$$S^{\infty}=(\prod_{\chi:S\to\mathbb{F}_p}\mathbb{F}_p)/\bigoplus_{I\subset S}\prod_{\chi:S/I\to\mathbb{F}_p}\mathbb{F}_p,$$
where $I$ ranges through all nonzero proper ideals of $S,$ and $\chi$ ranges through ring homomorphisms from $S$ to $\mathbb{F}_p.$
It follows immediately from Theorem \ref{weak-Chebotarev} that the natural homomorphism $S\to S^{\infty}$ is injective.
Note that $S^{\infty}$
 is a rational invariant of $S$: given a nonzero $f\in S$, then $S^{\infty}=(S_f)^{\infty}.$ 
 

Throughout given a Poisson algebra $\mathcal{O}$ over a ring $\bold{k}$, then by $P\Aut_{\bold{k}}(\mathcal{O})$ we denote
the group of Poisson $\bold{k}$-linear automorphisms of $\mathcal{O}.$

Given an $S$-algebra $A$, we put $$Z_{\infty}(A)=\prod_pZ(A/pA)/\bigoplus_pZ(A/pA).$$ So $Z_{\infty}(A)$ is a Poisson $S_{\infty}$-algebra.
One defines similarly $Z_{\infty}(A)$ for a $\mathbb{C}$-algebra $A.$ Also, we put 
$$Z^{\infty}(A)=\prod_{\chi:S\to\mathbb{F}_p}Z(A_{\mathbb{F}_p})/\bigoplus_{0\neq I\subset S}\prod_{\chi:S/I\to\mathbb{F}_p}Z(A_{\mathbb{F}_p}).$$
We have the following  natural restriction homomorphisms for a $\mathbb{C}$-algebra  $A$ (respectively an $S$-algebra $A$) 
$$Z_{\infty}:\Aut(A)\to P\Aut_{\mathbb{C}_{\infty}}(Z_{\infty}(A)), \quad Z^{\infty}:\Aut_S(A)\to P\Aut_{S^{\infty}}(Z^{\infty}(A)).$$
Next, we remark that if $A$ is a good quantization of $\mathcal{O}$ over $S$, then we have a natural homomorphism $$B\mathcal{O}_{S^{\infty}}\to Z^{\infty}(A)$$
defined by mapping $b\in \mathcal{O}$ to $\prod_{\chi:S\to \mathbb{F}_p}\iota_p(\bar{b}^p)$, where $\iota_p:\mathcal{O}_{\mathbb{F}_p}^p\to Z(A/pA)$ is the homomorphism
from Definition \ref{good-Qow} and $\bar{b}$ denotes the image of $b$ in $B_{\mathbb{F}_p}.$

 The next theorem is a key result of this section.
 It is motivated in part by 
 the proof of equivalence between the Jacobian and the Dixmier conjectures [\cite{BKK} Proposition 2],
except we have to work a little harder as we are working with more general quantizations than Weyl algebras, which usually are not
 Azumaya algebras in characteristic $p.$

Throughout, given a Noetherian domain $A$, we denote by $D(A)$ its skew field of fractions.

We also need to recall that given an algebraically closed field $\bold{k}$ and a $\bold{k}$-domain $A$ such that $A$ is finite over its center $Z$
and $Z$ is finitely generated $\bold{k}$-algebra, then the PI-degree of $A$ equals to the largest dimension of a simple $A$-module, and
the square of the PI-degree of $A$ equals to the rank of $A$ as a $Z$-module.

\begin{theorem}\label{center}
Let $S\subset\mathbb{C}$ be a finitely generated subring.
Let $S$-algebras $A, B$ be good filtered quantizations of affine Poisson $S$-algebras $\mathcal{O}, \mathcal{O'}$ respectively, such $\mathcal{O}_{\mathbb{C}}, \mathcal{O'}_{\mathbb{C}}$ are normal domains and $\dim(\mathcal{O})=\dim(\mathcal{O'}).$
Assume also that Poisson brackets on $\spec(\mathcal{O}), \spec(\mathcal{O'})$ are generically symplectic.
Let $\phi:A\to B$ be a $S$-algebra embedding. Then the homomorphism $\phi$
restricts to a homomorphism of $S^{\infty}$-Poisson algebras
$$Z^{\infty}(\phi):\mathcal{O}_{S_{\infty}}\to \mathcal{O'}_{S_{\infty}}.$$ 
Also, we have a restriction homomorphism
$$Z^{\infty}_A: \Aut_S(A)\to P\Aut(\mathcal{O}_{S^{\infty}})$$ which is nontrivial on semi-simple automorphisms.


\end{theorem}

In other words, the above defined $Z_{\infty}$ is a functor from the category
of good quantizations of Poisson normal varieties (which are generically symplectic of a fixed dimension) to the category of $\mathbb{C}_{\infty}$-Poisson
algebras. Moreover, the functor $Z_{\infty}$  is faithful when restricted to
either semi-simple automorphisms.

The statement of the above theorem becomes nicer for the following class of filtered quantization.
Their definition is directly motivated and related to the notion of Frobenius constant quantizations
introduced by Bezrukavnikov and Kaledin \cite{BK}.

\begin{defin}\label{Frobenius-constant}
Let $S\subset\mathbb{C}$ be a finitely generated subring.
Let $A$ be  a filtered quantization of a finitely generated $S$-graded Poisson algebra $\mathcal{O}.$ Then we call $A$ an
$\infty$-Frobenius constant quantization if for all large enough primes  $p\gg 0$, there is an isomorphism of Poisson algebras 
$$\iota:(\mathcal{O}/p\mathcal{O})^p\cong Z(A/pA)$$ which is the identity on the level of associated graded algebras, such
 that $\iota$ interchanges the reduction $\mod p$ Poisson bracket
on $Z(A/pA)$ with the negative of the Poisson bracket on $(\mathcal{O}/p\mathcal{O})^p.$

\end{defin}

Examples of $\infty$-Frobenius constant quantizations are: rings of differential operators
on smooth affine varieties, certain quantum Hamiltonian reductions with rational characters, in particular
central reductions $U_{\chi}(\mathfrak{g})$, where $\mathfrak{g}$ is a semi-simple Lie algebra and
$\chi:Z(U(\mathfrak{g}))\to\mathbb{C}$ is a central character corresponding to a simple finite dimensional representation.

Thus, given an injective homomorphism of $\infty$-Frobenius constant quantizations $i:A\to B$ of normal generically symplectic Poisson varieties $X, Y$ of equal dimension,
then by Theorem \ref{center} we get a canonical homomorphism of $\mathbb{C}_{\infty}$-Poisson varieties
$i_{\infty}: Y_{\mathbb{C}_{\infty}}\to X_{\mathbb{C}_{\infty}},$
as well as a canonical homomorphism 
$$\Aut(A)\to P\Aut(X_{\mathbb{C}_{\infty}}).$$

\begin{proof}[Proof of Theorem \ref{center}]

At first, we need to show that given an embedding $\phi:A\to B$ of good filtered quantizations $A, B$ over
a finitely generated ring $S\subset\mathbb{C}$ such that $A, B$ have the equal Gelfand-Kirillov dimension, then
$\phi_p(Z(A_p))\subset Z(B_p).$ 

Since $D(A), D(B)$ have equal Gelfand-Kirrilov dimensions,
it follows that $D(A)$ is a finite left (or right) $D(B)$-module via $\phi.$
Let $1\in V\subset B$ be a finite $S$-submodule, such that
$V$ generates $B$ as an $S$-algebra and $B\subset \phi(D(A))V.$
Let $1\in W\subset D(A)$ be a finite $S$-submodule so that $V^2\subset \phi(W)V.$
By localizing $S$ if necessary, we may assume that $V, W$ are spanned by elements whose coordinates belong to $S^*.$
Then given a base change $S\to\bold{k}$ to an algebraically closed field $\bold{k}$ of characteristic $p\gg 0,$
we have that $\Gr(A_{\bold{k}}), \Gr(B_{\bold{k}})$ are domains (as are $A_{\bold{k}}, B_{\bold{k}})$ and $W_{\bold{k}}\subset D(A_{\bold{k}}),$
as well as $\phi(W_{\bold{k}})\subset D(B_{\bold{k}}).$
Thus, $V_{\bold{k}}^2\subset \phi_{\bold{k}}(W_{\bold{k}})V_{\bold{k}}$, so $V^n_{\bold{k}}\subset  \phi_{\bold{k}}(W_{\bold{k}}^n)V_{\bold{k}}$ for all $n\geq 1.$
Denote by $A'$ the $\bold{k}$-subalgebra of $D(A_{\bold{k}})$ generated by $W_{\bold{k}}$. Thus, we have a homomorphism
$\phi_{\bold{k}}:A'\to D(B_{\bold{k}})$ and
 $B_{\bold{k}}\subset \phi_{\bold{k}}(A')V_{\bold{k}}$
with $\dim_{\bold{k}}V_{\bold{k}}<\infty.$ Hence the Gelfand-Kirillov dimension of $B_{\bold{k}}$ is at most the Gelfand-Kirillov dimension
of  $\phi_{\bold{k}}(A').$ Since the Gelfand Kirilov dimension of $A'$ is bounded above by the Gelfand-Kirillov dimension
of $A_{\bold{k}}$ which is equal to the Gelfand-Kirillov dimension of $\dim(B_{\bold{k}})$,
we conclude that $\phi_{\bold{k}}$ must be injective.

Suppose that for infinitely many primes $p$, there exists
 $z_p\in Z(A_p)$ 
such that $\phi_p(z)\notin Z(B_p).$ Then, there exists infinitely many primes $p$ and
 a base changes $S/pS\to\bold{k}$ to an algebraically closed field $\bold{k},$ so that $\Gr(A_{\bold{k}}), \Gr(B_{\bold{k}})$ are
domains and
$\phi_{\bold{k}}(Z(A_{\bold{k}}))\nsubseteq Z(B_{\bold{k}}).$ Let $z\in Z(A_{\bold{k}})$ so that $\phi_{\bold{k}}(z)\notin Z(B_{\bold{k}}).$
 Let $0\neq\delta\in Z(A_{\bold{k}})$ be such that $\delta$ vanishes on the complement of the Azumaya locus of $\spec(Z(A_{\bold{k}}))$. 
 Let $0\neq \delta'\in Z(B_{\bold{k}})$
be such that $\phi_{\bold{k}}(\delta)$ divides $\delta'$ and vanishes on the complement of the Azumaya locus of $\spec(Z(B_{\bold{k}})).$
Put $S_1=(A_{\bold{k}})_{\delta}$ and $S_2=(B_{\bold{k}})_{\delta'}.$
So, we have a homomorphism $\bar{\phi}:S_1\to S_2$ of Azumaya $\bold{k}$-algebras, such that $z\in Z(S_1)$ but $\phi(z)\notin Z(S_2).$
 
Next, by the assumption, we have that $\Gr(A_{\bold{k}})^p\subset \Gr(Z(A_{\bold{k}}))$. On the other
hand, since $\spec(\Gr(A_{\bold{k}}))$ is a normal Poisson variety which is symplectic on its smooth locus,
a very standard argument shows that $\Gr(Z(A_{\bold{k}}))\subset\Gr(A_{\bold{k}})^p$ (see for example [\cite{T1}, Lemma 2.4].) Thus, $\Gr(Z(A_{\bold{k}}))=\Gr(A_{\bold{k}})^p.$
Then the rank of $A_{\bold{k}}$ as a $Z(A_{\bold{k}})$-module equals to the rank of $\Gr(A_{\bold{k}})$ as a $\Gr(Z(A_{\bold{k}}))=\Gr(A_{\bold{k}})^p$-module
[\cite{T1}, Lemma 2.3].
On the other hand, since the PI-degree of $A_{\bold{k}}$ equals to the rank of $A_{\bold{k}}$ over its center, we conclude that
the PI-degree of $A_{\bold{k}}$ equals to $p^{\dim A_{\bold{k}}}$. Since $\dim A_{\bold{k}}=\dim B_{\bold{k}}$, we get that
$\text{PI-degree}(S_2)=\text{PI-degree}(S_1).$ Let $V$ be a simple $S_2$-module on which $\phi_{\bold{k}}(z)$ does not act like a scalar.
To construct such a module suffices to take a simple
module afforded by a character $\chi:Z(S_2)\to\bold{k}$ such that $\phi_{\bold{k}}(z)$ has a nonzero image in $(S_2/Z(S_2))_{\chi}.$
Then $V$ viewed as an $S_1$-module must be simple on which $z$ acts as a non-scalar,
a contradiction.




To summarize what we have so far, given an injective $S$-homorphism $f:A\to B$ (under the assumptions of the theorem), then it restricts to an $S^{\infty}$-Poisson homomorphism
$$Z^{\infty}(f):Z^{\infty}(A)\to Z^{\infty}(B).$$
Denote by $Z^{\infty}_n(A)\subset Z^{\infty}(A)$ the $S^{\infty}$-submodule consisting of images of
elements  of the form $\prod_{\chi:S\to \mathbb{F}_p}z_{\chi}$ with $\deg(z_{\chi})\leq n$ for all $\chi.$
Put $Z^{\infty}_b(A)=\bigcup_nZ^{\infty}_n(A).$ Clearly $Z_b^{\infty}(A)$ is an $S^{\infty}$-subalgebra of $Z^{\infty}(A).$
We define similarly $Z^{\infty}_{b}(B)$. Now, it follows easily that $f$ restricts to a homomorphism from $Z^{\infty}_b(A)$ to $Z^{\infty}_b(B).$
Recall that since $A$ is a good quantization of $\mathcal{O}=\Gr(A)$, we have a filtration preserving isomorphism of algebras $\mathcal{O}_{\mathbb{F}_p}^p\cong Z(A_{\mathbb{F}_p}).$
Next, we argue that $Z^{\infty}_b(A)$ is isomorphic to $\mathcal{O}_{S^{\infty}}.$
At first, remark that given a finite rank free $S$-module $V$, then
$V\otimes_SS^{\infty}=V^{\infty},$ where $$V^{\infty}=\prod_{\chi:S\to\mathbb{F}_p}V_{\mathbb{F}_p}/\bigoplus_{I\subset S}\prod_{\chi:S/I\to\mathbb{F}_p}V_{\mathbb{F}_p}.$$
So, we have isomorphisms $(\mathcal{O}_{\leq n})^p\cong Z_{pn}(A_{\mathbb{F}_p}), n\geq 0.$ Hence we get isomorphisms
$(\mathcal{O}_n)_{S^{\infty}}\cong Z_{pn}^{\infty}(A)$. Taking the union over $n,$ we get the desired isomorphism $\mathcal{O}_{S^{\infty}}\to Z_{b}^{\infty}(A)$. Therefore, we get the desired restriction homomorphism
of Poisson algebras $$Z^{\infty}(f):\mathcal{O}_{S^{\infty}}\to\mathcal{O'}_{S^{\infty}},$$ 
and by taking $A=B,$ we obtain the restriction homomorphism of the corresponding automorphism groups
$$Z^{\infty}_A:\Aut_{S}(A)\to P\Aut_{S^{\infty}}(\mathcal{O}_{S^{\infty}}).$$

Finally, suppose that $\Id\neq\phi\in\Aut_S(A)$ is a semi-simple automorphism which is in the kernel for $Z^{\infty}_A.$
Let $\chi:S\to\mathbb{F}_p$ be a homomorphism with $p\gg 0$, so that $\bar{\phi}=\chi(\phi)\neq \Id$ acts trivially on $Z(A_{\mathbb{F}_p})$ and $\Gr(A_{\mathbb{F}_p})$  (and hence $A_{\mathbb{F}_p}$ ) are domains. We also have that
$\bar{\phi}$ is a semi-simple automorphism over $\bar{\mathbb{F}}_p.$ 
Denote by $D$ the localization of $A_{\mathbb{F}_p}$ with respect to its center. So, $D$ is a division algebra over the field of fractions of
$Z(A_{\mathbb{F}_p}).$
  By the Skolem-Noether theorem, $\bar{\phi}$ is an inner automorphism of $D.$
Hence, there exists a non-zero element $a\in A_{\mathbb{F}_p}$ so that $\bar{\phi}(x)a=ax$ for all $x\in A_{\mathbb{F}_p}.$ 
If $y\in A_{\bar{\mathbb{F}}_p}$ is an eigenvector of $\phi$ of an eigenvalue $c\neq 1,$ then
$ay=cya$ which is a contradiction since $\Gr(A_{\bar{\mathbb{F}}_p})$ is commutative.

\end{proof}

\begin{remark}

The assumption in Theorem \ref{center} that $\phi:A\to A$ is semi-simple is essential. Indeed, let
$A=U(L),$ where $L$ is a 2-dimensional Lie algebra with a basis $h, x$ and the Lie bracket $[h, x]=x.$
Then put $\phi(a)=xax^{-1}, a\in A$, so $\phi(x)=x, \phi(h)=h-1.$ Therefore, $\phi$ is in the kernel of the restriction homomorphism
$$Z_{\infty}:\Aut(U(L))\to\prod_pZ(U(L_p)).$$
\end{remark}

\begin{proof}[Proof of Theorem \ref{automorphisms}]
Let $\Gamma\leq \Aut(A)$ be a finite subgroup.  We may choose a finitely generated subring $S\subset \mathbb{C}$ over which $A$ and the action of $\Gamma$ are defined.
So, $\Gamma\leq \Aut_S(A_S).$
Then by Theorem \ref{center}, we have an injective homomorphism $$Z^{\infty}|_{\Gamma}:\Gamma\to P\Aut(B_{S^{\infty}}).$$
Combining this with a homomorphism $S^{\infty}\to\mathbb{C}$, we obtain a subgroup $\Gamma'$ of
Poisson $\mathbb{C}$-automorphisms of $B,$ such that $\Gamma\cong \Gamma'.$ 
Recall that by the assumption $m$ is the unique Poisson maximal ideal of $B.$ So, $\Gamma'$ must preserve $m$, and moreover,
since the action of $\Gamma'$ on $B$ is semi-simple, it follows that the restriction homomorphisms $\Gamma'\to \Aut(m/m^2)$ is injective.
Hence, $\Gamma$ is isomorphic to a subgroup of the Lie algebra automorphisms of $m/m^2$, as desired.
\end{proof}

Now we can easily show Corollary \ref{W-algebra}.

\begin{proof}[Proof of Corollary \ref{W-algebra}]

Let $G\leq \Aut (W_{\chi}(\mathfrak{g}, e))$ be a finite subgroup of automorphisms of a central reduction (by a character $\chi$) of 
a finite W-algebra associated to a semi-simple Lie algebra $\mathfrak{g}$ and its nilpotent element $e\in\mathfrak{g}.$
By Lemma \ref{good-W}, $W_{\chi}(\mathfrak{g})$ is a good quantization of $S_e$-the Slodowy slice at $e$ of the nilpotent cone
of $\mathfrak{g}.$ Then we may use Theorem \ref{automorphisms} to conclude that G must be isomorphic to a subgroup
of Lie algebra automorphisms of $m/m^2$, where m is the maximal ideal in $\mathcal{O}(S_e)$ corresponding to the point $e\in S_e.$
Now, it is well-known that $m/m^2\cong \mathfrak{g}(e)$ and we are done. 
\end{proof}

Next, to show Theorem \ref{rigidity}, we in fact prove the following stronger result.

\begin{theorem}\label{uber-rigidity}
Let $X, Y,$ be affine normal conical Poisson varieties over $\mathbb{C}$, which are symplectic on their smooth loci.
Denote by $U, V$ the smooth loci of $X, Y.$
Let $A, B$ be good filtered quantizations of $X, Y$ respectively. Let $G\leq \Aut(A), G'\leq \Aut(B)$ be finite subgroups of $\mathbb{C}$-automorphisms and
let $\phi:A^G\to B^{G'}$ be a $\mathbb{C}$-algebra embedding. Suppose that  $B^{G'}$ is a finite left (or right) $A^{G}$-module
via $\phi.$ Then there exist groups $W, W',$ containing fundamental groups of $U, V,$ (respectively) as normal subgroups, such that
$$W/\pi_1(U)\cong G,\quad W'/\pi_2(V)\cong G',\quad W\cong W'.$$
 Moreover, if $\phi$ is an isomorphism then $W=W'.$

\end{theorem}

\begin{proof}

Let $S\subset\mathbb{C}$ be a large enough finitely generated subring over which everything is defined, and let
$\phi:A_S^G\to B_S^{G'}$ be an $S$-algebra embedding so that $B_S^{G'}$ is a finite left $A_S^G$-module
via $\phi.$ Then just as in \cite{T3} it follows that after further localizing $S$ if necessary, for any base change $S\to\mathbb{F}_p$
we have that $Z(A_{\mathbb{F}_p})^G=Z(A_{\mathbb{F}_p}^G)$ and $Z(B_{\mathbb{F}_p})^{G'}=Z(B_{\mathbb{F}_p}^{G'}).$
Now, by mimicking the proof of Theorem \ref{center}, we obtain an embedding of $S^{\infty}$-Poisson algebras
$$Z^{\infty}(\phi):\mathcal{O}(X_{S^{\infty}})^G\to\mathcal{O}(Y_{S^{\infty}})^{G'},$$ such that $\mathcal{O}(Y_{S^{\infty}})^{G'}$
is a finite $\mathcal{O}(X_{S^{\infty}})^G$-module. Combining $Z^{\infty}(\phi)$ with a base change $S^{\infty}\to\mathbb{C}$,
we obtain a finite morphism of affine Poisson varieties $\theta: Y/{G'}\to X/G.$

Let $U, U'$ denote the regular loci of $X, Y.$
Let $U ^{reg}$ (respectively $U'^{reg}$) denote the open subset of the regular locus of $X$ (respectively $X'$) consisting of points
with trivial stabilizers in $G$ (respectively $G'$). Then the codimension of $U\setminus U^{reg}$ in $U$ (respectively $U'\setminus U'^{reg}$)
is at least 2, so $\pi_1(U^{reg})=\pi_1(U)$ and $\pi_1(U'^{reg})=\pi_1(U').$
Let $Y_1, Y_2$ denote $U^{reg}/G, U'^{reg}/G'$ respectively. Then $Y_2\to Y_1$ is a finite covering map and $\pi_1(Y_1)$ (resp. $\pi_1(Y_2)$)
is an extension of $G$ (resp. $G')$ with the quotient $\pi_1(U)$. Now, the desired result follows.


\end{proof}

\begin{proof}[Proof of Theorem\ref{rigidity}]
Let $G, W\leq \Aut(A)$ be such that $A^G\cong A^{W}.$
Let $U$ be the smooth locus of $X.$ By Theorem \ref{uber-rigidity}, there is a group $\Gamma$ that contains
both $G, W$ as normal subgroups such that $\Gamma/G\cong \pi_1(U)\cong \Gamma/W.$ Therefore, if $|\pi_1(U)|<\infty,$
then $|G|=|W|$, in particular $A$ is rigid. If in addition $\pi_1(U)$ is trivial, then $G\cong W$ and we are done.

\end{proof}

\begin{remark}
We remark that in many (perhaps in most) interesting cases the algebraic fundamental group of the smooth locus of $X$ is indeed finite, for example
this is always the case when $X$ has symplectic singularity by a theorem of Namikawa \cite{N}.

\end{remark}

\begin{proof}[Proof of Theorem \ref{enveloping}]

Let $\phi:U(\mathfrak{g})\to U(\mathfrak{g'})$ be a $\mathbb{C}$-algebra
embedding with $\mathfrak{g}, \mathfrak{g}'$ semi-simple and
$\dim \mathfrak{g}=\dim \mathfrak{g'}.$
It is well-known that the maximal Krull dimension of a commutative subalgebra
of $U(\mathfrak{g'})$ is at most  $\frac{1}{2}(\dim(\mathfrak{g'})+\text{rank}(\mathfrak{g'})).$
On the other hand, $U(\mathfrak{g})$ contains a commutatve subalgebra of dimension  
$\frac{1}{2}\dim(\mathfrak{g}+\text{rank}(\mathfrak{g}))$ as proved by Rybnikov
\cite{R}. So we may conclude
 that $\rank(\mathfrak{g})\leq \text{rank}(\mathfrak{g'}).$
We may also assume that $\phi, \mathfrak{g}, \mathfrak{g}'$ are defined over $S$-a large enough finitely generated subring of $\mathbb{C}.$
Put 
$$Z(U\mathfrak{g})=S[g_1,\cdots, g_n], \quad Z(U(\mathfrak{g'}))=S[g'_1,\cdots, g'_n],\quad n=\rank(\mathfrak{g})=\rank(\mathfrak{g'}).$$
Denote by $\phi_{\bold{k}}:U(\mathfrak{g}_{\bold{k}})\to U(\mathfrak{g'}_{\bold{k}})$
 the base change of $\phi$ to an algebraically closed field $\bold{k}$ of characteristic $p\gg 0.$ Next we claim that, just as in the proof of
 \ref{center}, we have that
$\phi_{\bold{k}}(Z(U\mathfrak{g}_{\bold{k}}))\subset Z(U(\mathfrak{g'}_{\bold{k}}))$
for all $p\gg 0.$
Indeed, it is well-known that the PI-degree of $U(\mathfrak{g}_{\bold{k}})$
equals to $p^{\frac{1}{2}(\dim \mathfrak{g}-\text{rank}(\mathfrak{g}))}.$
Thus 
$$\text{PI}-\deg(U(\mathfrak{g}_{\bold{k}}))\geq \text{PI}-\deg(U(\mathfrak{g'}_{\bold{k}})).$$
After this, the proof of Theorem \ref{center} carries over word by word to yeild 
$$\phi_{\bold{k}}(Z(U\mathfrak{g}_{\bold{k}}))\subset Z(U(\mathfrak{g'}_{\bold{k}})),\quad \rank(\mathfrak{g})=\rank(\mathfrak{g'}).$$
Next, we recall that by Veldkamp's theorem, we have
$$Z(U(\mathfrak{g'}_p))=Z_p(\mathfrak{g'}_p)[g'_1,\cdots, g'_n].$$
So, for $z\in S[g_1,\cdots, g_n],$ we have that $\phi_p(\bar{z})\in Z_p(\mathfrak{g'}_p)[g'_1,\cdots, g'_n]$ (for all $p\gg 0.)$
This implies that $\phi(z)\in S[g'_1,\cdots, g'_n]$, so $\phi(Z(U\mathfrak{g}))\subset Z(U(\mathfrak{g'})).$
Let $\chi: Z(U(\mathfrak{g}))\to \mathbb{C}$ be a character. Let $\chi':Z(U(\mathfrak{g'}))\to \mathbb{C}$ be any character such that 
$\phi(\ker(\chi))\subset\ker(\chi').$ Then we get the homomorphism $\phi:U_{\chi}(\mathfrak{g})\to U_{\chi'}(\mathfrak{g'}).$
Hence, it suffices to prove the second part of the theorem.

So, let $\phi:U_{\chi}(\mathfrak{g})\to U_{\chi'}(\mathfrak{g'})$ be a $\mathbb{C}$-algebra homomorphism.
Let $\mathcal{N}_{\mathfrak{g}}$, respectively $\mathcal{N}_{\mathfrak{g'}}$ denote the nilpotent cone
of $\mathfrak{g}^*$ (resp. $\mathfrak{g}'^*).$
As $U_{\chi}(\mathfrak{g})$ (respectively $U_{\chi'}(\mathfrak{g}')$) are good quantizations of
$\mathcal{O}(\mathcal{N}_{\mathfrak{g}})$ (resp. $\mathcal{O}(\mathcal{N}_{\mathfrak{g'}})$), it follows
from Theorem \ref{center} that there is a Poisson algebra homomorphism
$\psi: \mathcal{O}(\mathcal{N}_{\mathfrak{g}})\to \mathcal{O}(\mathcal{N}_{\mathfrak{g'}}).$
In particular, we get a Lie algebra homomorphism $\psi:\mathfrak{g}\to \mathcal{O}(\mathcal{N}_{\mathfrak{g'}}).$
Denote by $m$ the maximal ideal of $\mathcal{O}(\mathcal{N}_{\mathfrak{g'}})$ corresponding to the origin-the augmentation
ideal.
Since $[\mathfrak{g}, \mathfrak{g}]=\mathfrak{g}$, it follows that $\psi(\mathfrak{g})\subset m.$
 We may identify $m/m^2$ with $\mathfrak{g'}.$ Then we have a homomorphism of Lie algebras
$$\bar{\psi}:\mathfrak{g}\to m/m^2=\mathfrak{g'}.$$ Let $I=\ker(\bar{\psi}).$ Thus $[I, I]=I.$
Let $\psi(I)\subset m^n$ for $n>1.$ We have
$$\psi(I)=\psi([I, I])\subset \lbrace m^n, m^n\rbrace\subset m^{2n-1}.$$
In particular $I\subset m^{n+1}.$ Arguing by induction on $n,$ we conclude that 
 $I=0.$ Hence $\mathfrak{g}$ is isomorphic to a Lie subalgebra of $\mathfrak{g}'$ and we are done.

\end{proof}

\section{Cherednik algebras as fixed rings}

Given a simple domain $B$ over $\mathbb{C}$, it is an interesting and natural problem to classify
finite groups $\Gamma$ for which there exists a domain $R$ on which $\Gamma$ acts faithfully via $\mathbb{C}$-algebra automorphisms,
such that $B=R^{\Gamma}.$  Given the direct analogy with Galois theory, 
we refer to this question as the inverse Galois problem for $B.$
In \cite{T3} we solved this problem
for rings of differential operators on smooth affine varieties. Namely,
 if $D(X)=R^{\Gamma},$
where $X$ is a smooth affine variety and $\Gamma$ is a finite group of $\mathbb{C}$-automorphisms of a domain $R,$
then there exists a smooth affine variety $Y$ such that $R\cong D(Y)$ and $Y\to X$ is a $\Gamma$-Galois etale covering of $X$
[\cite{T3}, Theorem 1]. In particular, in the spirit of Galois theory, there is a bijection between
normal subgroups of $\Gamma$ and $\Gamma$-invariant subalgebras of $R$ containing $D(X).$
It was also shown in [\cite{T3}, Theorem 2] that very generic central quotients of enveloping algebras of semi-simple Lie algebras
cannot be nontrivial fixed rings. In this section we apply the methodology of \cite{T3} to the case when $B$ is a (simple)
 spherical subalgebra of a rational Cherednik algebra defined by Etingof and Ginzburg \cite{EG}. Let us recall their definition.

 Let $W$ be a complex reflection group, $\mathfrak{h}$ its reflection representation
and $S\subset W$  the set of all complex reflections.
Let $(\, ,\,):\mathfrak{h}\times \mathfrak{h}^*\to \mathbb{C}$ be the natural pairing. Given a reflection $s\in S,$
let $\alpha_s\in \mathfrak{h}^*$ be an eigenvector of $s$ for eigenvalue $1$.
Also, let $\alpha_s^{\vee} \in \mathfrak{h}$ be an eigenvector normalized so that $\alpha_s(\alpha_s^{\vee} )=2.$
Let $c:S\to \mathbb{C}$ be a function invariant with respect to conjugation with $W.$ 
The rational Cherednik algebra $H_{c}$ associated to $(W, \mathfrak{h})$ with parameter
$c$ is defined as the quotient of $\mathbb{C}[W]\ltimes T( \mathfrak{h}\oplus \mathfrak{h^*})$ by the following relations
$$
[x, y]=(y,x)-\sum_{s\in S}c(s)(y,\alpha_s)(\alpha_s^{\vee} , x),\quad [x, x']=0=[y,y']
$$
for all $x, x'\in \mathfrak{h}^*$ and $y, y'\in \mathfrak{h}.$

We are concerned with the spherical subalgebra $B_c$ of the Cherednik algebra $H_c.$
Recall that  $$B_c=eH_ce, e=\frac{1}{|W|}\sum_{g\in W}g.$$
For $c=0$, we have that $B_0=D(\mathfrak{h})^W.$

Next, we recall a couple of basic properties of spherical subalgebras of rational Cherednik algebras.
Namely the PBW property and the Dunkl isomorphism.

The crucial PBW property of $H_c, B_c,$ implies that if we equip $H_c, B_c,$
with an algebra filtration by putting 
$$deg(\mathfrak{h})=1,\quad deg(\mathfrak{h^*})=1,\quad deg(W)=0,$$ then
$$\Gr(H_c)=\mathbb{C}[W]\ltimes \Sym(\mathfrak{h}\oplus \mathfrak{h^*}), \quad \Gr(B_c)=\Sym(\mathfrak{h}\oplus \mathfrak{h^*})^{W}.$$
Recall that since for any nonzero $f\in \Sym(\mathfrak{h}^*), \text{ad}(f)=[f, -]$ acts locally nilpotently on $H_c,$
we may consider the localization $H_c[f^{-1}]$ (and $B_c[f^{-1}]$ for $f\in  \mathbb{C}[\mathfrak{h}]^W$).
Then we have the induced filtration on $B_c[f^{-1}]$ and 
$$\Gr(B_c[f^{-1}])=\Sym(\mathfrak{h}\oplus \mathfrak{h^*})^{W}_f.$$

Set $\mathfrak{h}^{reg}=\lbrace x\in \mathfrak{h}, (x, \alpha)\neq 0, \alpha\in S\rbrace.$
Let $\delta \in \mathbb{C}[h]^W$ be the defining function of  $\mathfrak{h}\setminus\mathfrak{h}^{reg}.$ 
 Recall that via the Dunkl embedding we have an isomorphism 
 $$B_c[\delta^{-1}]\cong D(\mathfrak{h}^{\text{reg}}).$$


The next theorem is the main result of this section. It relates the inverse Galois problem for $B_c$
to the geometry of the center of reduction of $B_c$ modulo a large prime. 
In fact, it applies to a wider class of algebras to be defined below.

\begin{assumption}\label{assumption}

Let $A$ be an affine  Noetherian $\mathbb{C}$-domain, let
$V\subset A$ be a finite dimensional $\mathbb{C}$-subspace, such that the following holds.
For any $g\in V$, the adjoint action $\ad(g)$ is locally nilpotent and there exists $0\neq\delta\in V$ so that
$A_{\delta}$ (the localization of $A$ with respect to $\delta$) can be identified with $D(X),$ where $X$ is a smooth affine variety over $\mathbb{C}$
and $V\subset\mathcal{O}(X).$ Moreover, 
there exists a finitely generated subring $S\subset\mathbb{C}$ and an $S$-model $A_S$ of $A$, such that
$A_{\bold{k}}$ is a finite module over its center for any
base change $S\to \bold{k}$ for any field $\bold{k}$ of large enough characteristic. 

\end{assumption}

Next, we define Harish-Chandra bimodules over algebras satisfying the above assumption.
We also recall the definition of Harish-Chandra bimodules over spherical subalgebras of rational Cherednik algebras.

\begin{defin}

Let $A$ be an algebra from Assumption \ref{assumption}. Then an $A$-bimodule $M$ is said to be a Harish-Chandra bimodule if
for any $g\in V,$ the adjoint action $\ad(g)|_M$ is locally nilpotent.
Let $B_c$ be a spherical subalgebra associated to $(W, \mathfrak{h}).$
Let $M$ be a bimodule over $B_c$. Then $M$ is said to be a Harish-Chandra bimudule if $\ad(x)$ is locally nilpotent on $M$
for any $x\in\mathbb{C}[\mathfrak{h}]^W, \mathbb{C}[\mathfrak{h^*}]^W.$

\end{defin}

\begin{theorem}\label{HC}

Let $B$ be a $\mathbb{C}$-algebra satisfying Assumption \ref{assumption}. Suppose that $B$ is simple.
Let $B=R^{\Gamma}$, where $R$ is a $\mathbb{C}$-domain and $\Gamma$ is a finite subgroup of $\mathbb{C}$-algebra
automorphisms of $R.$
 Then $R$ is a Harish-Chandra $B$-bimodule and satisfies Assumption \ref{assumption}.
There exists a finitely generated ring $S\subset\mathbb{C} $ over which everything is defined, such that the following holds.
For any base change $S\to\bold{k}$ to an algebraically closed field of positive characteristic
the group $\Gamma$
is a quotient of the etale fundamental group of the Azumaya locus of
$\spec Z(B_{\bold{k}}).$ 

\end{theorem}

\begin{proof}
Suffices to show this when $\Gamma$ is simple.
Since $B$ is a simple Noetherian domain such that its localization is isomorphic to the ring of differential operators
on a smooth affine variety, it follows
 that $Z(B)=\mathbb{C}.$ Now, using some
  standard facts about fixed rings \cite{M}, we can deduce that $B$ is Morita equivalent to the skew ring $\mathbb{C}[\Gamma]\ltimes R$ (see [\cite{T3}, Lemma 4]).
There exists a large enough finitely generated ring $S\subset\mathbb{C}$,
and models of $B, R$ over $S$, to be denoted by $B_S, R_S,$
 so that $B_S$ is Morita equivalent
to $S[\Gamma]\ltimes R_S.$ In particular, $R_S$ is a projective left (and right) $B_S$-module. So for large enough $p\gg 0$ and a base change $S\to \bf{k}$ to an algebraically
closed field of characteristic $p$, we have that $B_{\bf{k}}$ is Morita equivalent to
$\bold{k}[\Gamma]\ltimes R_{\bf{k}}.$ We also remark that $R_{\bf{k}}$ is a torsion free $B_{\bf{k}}$-module.

Next, we claim that $Z(R_{\bold{k}})^{\Gamma}=Z(B_{\bold{k}}).$ Indeed, since $Z(R_{\bold{k}})^{\Gamma}\subset B_{\bold{k}}\cap Z(R_{\bold{k}}),$ it follows that $Z(R_{\bold{k}})^{\Gamma}\subset Z(B_{\bold{k}}).$
On the other hand, since $Z(B_{\bold{k}})=Z(\bold{k}[\Gamma]\ltimes R_{\bf{k}})$ (by the above mentioned Morita equivalence),
it follows that $Z(B_{\bold{k}})$ commutes with $R_{\bold{k}}.$ Since $Z(B_{\bold{k}})\subset R_{\bold{k}}^{\Gamma}$, we conclude that
$Z(B_{\bold{k}})\subset Z(R_{\bold{k}})^{\Gamma}$ and we are done.
 
By our assumptions $B_{\bold{k}}$ is finite over its center, and $Z(B_{\bold{k}})$ is a domain. 
 Let $f\in Z(B_{\bold{k}})$ be a nonzero element such that it vanishes on complement of the Azumaya locus of
$\spec(Z(B_{\bold{k}})).$ As $f$ is also central in $R_{\bold{k}}$, we may localize $B_{\bold{k}}, R_{\bold{k}}$
at $f$ to be denoted respectively by $(B_{\bold{k}})_f, (R_{\bold{k}})_{f}.$

We have that
$(B_{\bold{k}})_f$ is an Azumaya algebra over $Z(B_{\bold{k}})_f$ and $(B_{\bold{k}})_f$
is Morita equivalent to $\bold{k}[\Gamma]\ltimes (R_{\bf{k}})_f.$
Then just as in [\cite{T3}, Proposition 1], we can conclude that 
$\spec Z(R_{\bold{k}})_f\to \spec Z(B_k)_f$ is
a $\Gamma$-Galois etale covering and
$$(R_{\bold{k}})_f=B_{\bold{k}}\otimes_{Z(B_{\bold{k}})} Z(R)_f.$$
Let $U$ denote the Azumaya locus of $\spec(Z(B_{\bold{k}})),$ and $Y$ denote
the preimage of $U$ under the projection $\spec(Z(R_{\bold{k}}))\to \spec Z(B_{\bold{k}}),$
then $Y\to U$ is $\Gamma$-Galois covering.
In particular, for any $g\in V_{\bold{k}}$, $\text{ad}(g)$ acts locally nilpotently on $(R_{\bold{k}})_f.$ Therefore $\text{ad}(g)$ acts locally
nilpotently on $R_{\bold{k}}$ as $R_{\bold{k}}$ is $Z(B_{\bold{k}})$-torsion free (since $R_{\bold{k}}$ is projective over $B_{\bold{k}}$).
Now it follows that if $R_{\bold{k}}$ is a domain (which is proved in the next paragraph), then $\Gamma$ must be a quotient of the etale fundamental group
of the smooth locus of $\spec(Z(B_{\bold{k}}))=U.$ 

Next, we show that $R$ is a Harish-Chandra $B$-bimodule and $R_{\bold{k}}$ is a domain, which completes the proof.
Let $g\in V, z\in R.$ We want to show that $\ad(g)^m(z)=0$ for some $m.$
It suffices to check that there exists $m$, such that $\ad(g)^m(z)=0$ in $R_{\bold{k}}$ for all base changes $\bold{k}$ of large enough
characteristic. Let $$z^l +\sum_{i<l}{a_i}z^i=0,\quad a_i\in B.$$  Recall that by the assumption,  we may identify $B_{\delta}$ with
$D(X)$. Denote by $R'_{\bold{k}}$ the localization of $R_{\bold{k}}$ with respect to $\delta.$
Let $m$ be the largest of orders of $a_i$ viewed as differential operators in $D(X).$ 
We have that $\bold{k}[\Gamma]\ltimes R'_{\bf{k}}$ is Morita equivalent to $D(X_{\bold{k}}).$
Then by [\cite{T3}, Theorem 1], there exists a $\Gamma$-Galois covering $Y\to X_{\bold{k}}$
such that $R'_{\bold{k}}\cong D(Y).$ 
Since the images of $a_i$ in $R'_{\bold{k}}\cong D(Y)$ are differential operators of order at most $m$,
it follows easily from the above equality that the degree of $z$ viewed as a differential operator in $D(Y)$ is at most $m.$
So, $\ad(g)^m(z)=0$ in $R_{\bold{k}}$ as desired. Thus, $R$ is a Harish-Chandra $B$-bimodule.

Finally,
it follows that $ad(\delta)$ acts locally nilpotently on $R.$ 
Let $R'=R[\delta^{-1}].$  Since $R'^{\Gamma}=D(X)$
it follows from [\cite{T3}, Theorem 1] that $R'\cong D(Z)$ for some smooth affine variety $Z$ equipped
with a $\Gamma$-Galois covering $Z\to X.$
Hence $R'_{\bf{k}}$ is a domain for $\text{char}(\bold{k})\gg 0$, as desired.

\end{proof}

 Next, we apply Theorem \ref{HC} to classes of Cherednik algebras for which the center of their reduction modulo
$p>0$ is well-known and (relatively) easy to describe. Namely, we consider two families of  spherical subalgebra
of the rational Cherednik algebras: one  associated to the pair $(S_n, \mathbb{C}^n)$ and a parameter $c\in\mathbb{C},$
the other is noncommutative deformations of Kleinian singularities of type A.

\begin{theorem}\label{Sn}

Let $B_c$ be the spherical subalgebra of a rational Cherednik algebra associated with $(S_n, \mathbb{C}^n)$ with a parameter $c\in\mathbb{C}.$
Assume that $B_c$ is simple.
If $c$ is irrational, then $B_c$ cannot be a fixed ring of a domain under a nontrivial  action of a finite
group of ring automorphisms. For  rational $c$, if $B_c=R^\Gamma$ with finite group $\Gamma$ and domain $R,$
then $\Gamma$ must be a quotient of $S_n.$

\end{theorem}

To use Theorem \ref{HC}, we need to know the $p'$-part of the etale fundamental group
of the smooth locus of $\spec(Z(B_{\bar{c}})).$ For this purpose we utilize the following.

\begin{remark}\label{lifting}
Let $X$ be a complete smooth variety over an algebraically closed field $\bf{k}$ of characteristic
$p,$ and $U\subset X$ be an open subset such that $X\setminus U$ is a divisor with normal crossings
in $X.$ Let $\tilde{X}$ be a complete smooth lift of $X$ over $W(\bold{k})$ ($W(\bold{k})$ is the ring of Witt vectors over
$\bold{k}$), $\tilde{U}\subset \tilde{X}$
be an open subset lifting $U$, such that $\tilde{X}\setminus\tilde{U}$ is a divisor with normal crossings over $W(\bold{k})$.
Then any $p'$-degree Galois covering of $U$ admits a lift to a Galois covering of $\tilde{X}$ [\cite{LO}, Corollary A.12], which yields
that any $p'$-quotient of the etale fundamental group of $U$ must be a quotient of the fundamental group of
$U_{\mathbb{C}}.$
 
\end{remark}

We need the following corollary of the Chebotarev density theorem.
It contains slightly more than [\cite{VWW}, Theorem 1.1]. We present a short proof for a reader's convenience.

\begin{lemma}\label{cheb}

Let $S$ be a finitely generated domain containing $\mathbb{Z}$ and $c\in S.$
Then there are infinitely many primes $p$ and ring homomorphisms $\phi_p:S\to \mathbb{F}_p.$
If $c\notin \mathbb{Q}$ then there exists infinitely many primes $p$ and homomorphisms
$\phi_p:S\to F_q$, so that $\phi_p(c)\notin \mathbb{F}_p$ and $q$ is a power of $p.$
\end{lemma}
\begin{proof}
By the Noether normalization theorem, there exists $l\in \mathbb{N}$ and
algebraically independent $x_1,\cdots, x_n\in S_l$ so that
$S_l$ is integral over $\mathbb{Z}_l[x_1,\cdots, x_n].$ Let $I$ be a prime ideal in $S_l$ laying over $(x_1,\cdots, x_n)$
(such ideal exists since $\spec(S)\to \spec(S_l[x_1,\cdots, x_n])$ is surjective by the going-up theorem).
 So, $S_l/I=R$ is an integral domain finite over $\mathbb{Z}_l.$
Let $S'$ be the integral closure of $\mathbb{Z}$ in $R.$ Then $R=S'_l.$ 
Thus suffices to show that there exists a homomorphism $\phi_p:S'\to \mathbb{F}_p$ for infinitely many $p$. This is a consequence of the Chebotarev density theorem.

  We have that the image of the map $\spec(S)\to \spec \mathbb{Z}[c_1]$ contains a nonempty open subset.
If $c_1$ is algebraic, then all but finitely many prime ideals in $\mathbb{Z}[c_1]$ lift to $S.$
By the Chebotarev denisity theorem there are infinitely many primes $I\subset \mathbb{Z}[c_1]$
such that the image of $c_1$ in the quotient  $\mathbb{Z}[c_1]/I\cong F_q$ does not belong to $\mathbb{F}_p.$
Let $I'\in\spec(S)$ be a lift of $I.$
Now any homomorphism $S/I'\to \bar{F_p}$ lifting $\mathbb{Z}[c_1]/I\to F_q$ will do.
Finally, let $c_1$ be transcendental. Let $f\in \mathbb{Z}[c_1]$ be such that $\spec (\mathbb{Z}[c_1]_f)$
lifts to $\spec(S).$ Thus it suffices to show that there are infinitely many primes $p$ for which there exists
$t\in\mathbb{Z}[c_1]$ such that $f\notin (p, t)$ and $\mathbb{Z}[c_1]/(p, t)=F_q$ for $q>p.$ 
For this purpose we can take any $p$ that does not divide $f$, then take a nonlinear irreducible $\bar{t}\in F_p[c_1]$
that does not divide $f\mod p.$ Then let $t$ be any lift of $\bar{t}.$
\end{proof}

For the proof of Theorem \ref{Sn} we need to recall the definition of the $n$-th Calogero-Moser space. Consider the following
subscheme of pairs of $n$-by-$n$ matrices over $\mathbb{C}$

$$X=\lbrace (A, B)|\quad \text{rank}([A, B]+\text{Id}_n)=1\rbrace.$$
It is known that $PGL_n(\mathbb{C})$ acts freely on $X$ by conjugation, and the $n$-th Calogero Moser space,
denoted by $\text{CM}_n,$ is defined as the quotient 
$$X//PGL_n(\mathbb{C})=\text{CM}_n.$$
It is well-known that $\text{CM}_n$ is a smooth, affine variety over $\mathbb{C}$ \cite{W}.
In the following proof, we also need that the Calogero-Moser spaces are simply connected.
This follows from the fact that the $n$-th Calogero-Moser space is homeomorphic to
the Hilbert scheme of $n$-points on the plane which is known to be simply connected based on its cell decomposition.

\begin{proof}[Proof of Theorem \ref{Sn}]

If $c$ is rational then 
after a base change to a field $\bold{k}$ of characteristic $p,$ we have that \cite{BFG}
$$\spec(Z(B_{\bar{c}}))=(\mathfrak{h}\oplus \mathfrak{h}^*)/S_n.$$ 
  Hence using Remark \ref{lifting}, the  $p'$-etale fundamentale group
of the smooth locus of $\spec(Z(B_{\bar{c}}))$ is $S_n.$ For irrational $c,$ by Lemma
\ref{cheb} for any finitely generate subring $S\subset \mathbb{C},$ there are infinitely many primes $p$ and algebraically close fields $\bf{k}$ of characteristic $p$
with a base change $S\to\bold{k}$, such that $\bar{c}\notin \mathbb{F}_p.$ Then as explained in \cite{BFG}, we have 
$$\spec Z(B_{\bar{c}})\cong (\text{CM}_{n})_{\bold{k}}.$$
Since $(\text{CM}_{n})_{\bold{k}}$  admits a smooth
simply connected lift to characteristic 0 (namely $\text{CM}_n$), using Remark \ref{lifting} as before the desired assertion follows.

\end{proof}

\begin{remark}

Given a Cherednik algebra $H_c$ associated with an arbitrary pair $(W, \mathfrak{h})$, we expect that
its spherical subalgebra $B_c$ is a good quantization of $\mathfrak{h}\oplus\mathfrak{h}^*/W.$
Then Theorem \ref{HC} would imply that if $B_c=R^{\Gamma},$ where $B_c$ is simple
and $R$ is a domain, then $\Gamma$ must be a quotient of $W.$

\end{remark}

\section{The case of generalized Weyl algebras}

In this section we apply results of the previous one to
  noncommutative deformations
of the Kleinian singularities of type A (as introduced by Hodges \cite{H}), which are spherical subalgebras of rational Cherednik algebras
associated with the pair a cyclic group and its one dimensional representation. This extensively studied family of algebras is also known as (classical) generalized Weyl algebras. 
Let us recall their definition.

Let $v=\prod_{i=1}^n (h-t_i)\in \mathbb{C}[h].$
Then the algebra $A(v)$ is generated by $x, y, h$ subject to the relations
$$xy=v,\quad yx=v(h-1),\quad hx=x(h+1),\quad hy=y(h-1).$$
If $\deg(v)=1$, then $A(v)$ is isomorphic to the first Weyl algebra $A_1(\mathbb{C}).$
Recall also that if $v=\prod_{i=0}^{n-1}(h+\frac{i}{n}),$ then $A(v)$ can be identified with the fixed ring
of the Weyl algebra $A_1(\mathbb{C})$ under the natural action of the cyclic group of order $n.$
On the other hand, when $\deg(v)=2$ algebras $A(v)$ correspond to central quotients of the enveloping algebra $U(\mathfrak{sl}_2).$
We next recall that algebra $A(v)$ is simple if and only if roots of $v$ differ by non-integers.

It was observed by Smith \cite{S} that a countable family of primitive quotients
of $U(\mathfrak{sl}_2)$ can be realized as $\mathbb{Z}/2\mathbb{Z}$-fixed rings
of algebras of differential operators on singular algebraic curves $\spec(\mathbb{C}[x, z]/(z^2-x^m))$, where $m>1$ is odd. 
On the other hand, it follows from our earlier result \cite{T3} that the first Weyl algebra (in fact any $n$-th Weyl algebra) cannot be a nontrivial
fixed ring of a domain.
Naturally, one wonders which other generalized Weyl algebras can be realized as nontrivial fixed rings.
Some sufficient conditions for a generalized Weyl algebra to be a nontrivial fixed ring of another generalised Weyl algebra were obtained in \cite{JW}, \cite{GW}.

Our next result fully solves the inverse Galois problem for simple generalized Weyl algebras. To state it, we need to introduce a certain class of algebras
which incorporates generalized Weyl algebras, as well as rings of differential operators on singular affine curves 
$\spec(\mathbb{C}[z, x]/(z^l-x^m))$ with $m=1\mod l.$

Let $v\in\mathbb{C}[h]$ and $l, m$ be coprime natural numbers. Then we have a derivation (the Euler vector field) $D$ on the ring $\mathcal{O}=\mathbb{C}[x, z]/(z^l-x^m)$
defined as follows 
$$D(x)=x, D(z)=\frac{m}{l}z.$$
 Putting $[h, -]=D$ we can define the semi-direct product algebra
$A=\mathbb{C}[h]\ltimes\mathcal{O}.$ Put $y=x^{-1}v$ considered as an element of $A[x^{-1}]=\mathbb{C}[h]\ltimes\mathcal{O}[x^{-1}].$
Then the subalgebra generated by $A$ and $y$ is denoted by $A_v^{l, m}.$ Clearly, $A(v)$ is a subalgebra of $A_v^{l, m}.$
Moreover, $\mathbb{Z}/l\mathbb{Z}=\langle \sigma\rangle$ acts on $A_v^{l, m}$ by
$$\sigma(z)=\xi z, \sigma(x)=x, \sigma(h)=h,$$ where $\xi$ is a primitive $l$-th root
of unity.

\begin{theorem}\label{typeA}
Let $A(v)$ be simple with $\deg(v)=n.$ If $A(v)=R^{\Gamma}$ with $R$ a $\mathbb{C}$-domain and $\Gamma$ a finite group of $\mathbb{C}$-algebra automorphisms of $R,$
then $\Gamma$ must be a quotient of $\mathbb{Z}/n\mathbb{Z}.$ 
Let $\Gamma=\mathbb{Z}/l\mathbb{Z}, l|n.$ Then $A(v)=R^{\Gamma}$ for some domain $R$ if and only if the set of images of roots
of $v$ in $\mathbb{C}/\mathbb{Z}$ is closed under the shift by $1/l.$ In this case $R\cong A_v^{l, m}$ for some $m=1\mod l.$

\end{theorem}

\begin{proof}

We know by Theorem \ref{HC} that $R$ must be a Harish-Chandra $A(v)$-bimodule.
Since $A(v)[x^{-1}]$ can be identified with $D(\mathbb{C}[x, x^{-1}]),$ it follows from [\cite{T3}, Theorem 1]
that there exists a an affine variety $Y$ and a $\Gamma$-Galois covering $$Y\to \spec \mathbb{C}[x, x^{-1}]=\mathbb{C}^*,$$
so that we have a $\Gamma$-equivariant isomorphism $R[x^{-1}]\cong D(Y).$
So, $Y\cong \mathbb{C}^*$ and $\Gamma$ must be a cyclic group. Let $l=|\Gamma|.$
Therefore, $R[x^{-1}]\cong D(\mathbb{C}[w, w^{-1}]),$
and $w^l=x$ with $\Gamma=\langle \sigma\rangle$ acting on $w$
by the multiplication by $\xi,$ where $\xi$ is a primitive $l$-th root of unity. A similar statement holds for $R[y^{-1}].$
We may write $w=x^{-k}z$ for some $k> 0, z\in R.$ So, $z^l=x^m$ for some $m>0$ so that $m=1\mod l$ and $\sigma(z)=\xi z.$
Denote by $B$ the subalgebra of $R$ generated by $A(v)$ and $z.$ We claim that $B=R.$
Indeed, $\Gamma$ acts faithfully on $B.$
 Therefore, $B[x^{-1}]=R[x^{-1}]$ and $B[y^{-1}]=R[y^{-1}].$ Put $M=R/B$. then $M$ is an $A(v)$-module such that
 $M[x^{-1}]=0$ and $M[y^{-1}]=0.$ This easily implies that $M$ must be a finite dimensional $A(v)$-module. Since
 $A(v)$ admits no nonzero finite dimensional modules (as it is simple), it follows that $M=0,$ hence $B=R.$
 Next, we may and will identify $R=B$ with $A_v^{l, m}$ inside the algebra 
 $$R[x^{-1}]=\mathbb{C}[h]\ltimes \mathbb{C}[z, x, x^{-1}]/(z^l-x^m).$$
Thus, it remains to establish when does the equality $(A_v^{l,m})^{\sigma}=A(v)$ hold. 
 

 Put for simplicity $A=A(v).$ So, $z^{l-1}Az\subset A.$ Recall that
$zhz^{-1}=h-\frac{m}{l}.$ It is well-known and easy to check that $x^my^m=v^{[m]}$, where $v^{[n]}=\prod_{i=0}^{n-1}v(h-i).$
 So, $y^n=x^{-n}v^{[n]}$. Hence $$z^{l-1}y^nz=x^{-n}z^l(z^{-1}v^{[n]}z)=x^{m-n}v^{[n]}(h+\frac{m}{l}).$$
 Take $n=m+1.$ Then $$z^{l-1}y^{m+1}z=x^{-1}v^{[m+1]}(h+\frac{m}{l}).$$
Now recall that we have a standard $\mathbb{Z}$-grading on $A[x^{-1}]$ defined as follows: $\deg(x)=1$ and $\deg(h)=0.$ 
  Since the element  $x^{-1}v^{[m+1]}(h+\frac{m}{2})$  has degree -1, it
 follows that this element must equal to $ya$ for some $a\in\mathbb{C}[h].$ Hence $v^{[m+1]}(h+\frac{m}{l})$ divides $v.$
Write $v=\prod_k (h-t_k).$ Thus, since roots of $v$ differ by non-integers, it follows that for each root $t_k$ there exists another
root $t_{k'},$ such that $$t_{k'}-t_k=\frac{m}{l}+i,\quad i\in \mathbb{Z}.$$ Hence the set of images of roots of $v$ in $\mathbb{C}/\mathbb{Z}$
is closed under the shift by $\frac{m}{l},$ as desired.

Now, we assume that the set of images of roots of $v$ in $\mathbb{C}/\mathbb{Z}$
is closed under the shift by $\frac{1}{l}$ and show that $(A_v^{l, m})^{\sigma}=A$ for $m\gg 0.$ 
For simplicity we assume that $l=2,$ the general case is similar.
Assume $m$ is given so that for
any root $t$ of $v,$ there exists another root $t'$ of $v,$ so that $t-t'=\frac{i}{2}$ for some odd $i$ with $|i|\leq m.$
We claim that $(A_v^{2, m})^{\sigma}=A.$ We need to show that $zA_{v}^{2, m}z\subset A$, for which it suffices to check
that $zy^kz\in A$ for all $k.$ Indeed, recall that  $$zy^{k}z=x^{-k+m}v^{[k]}(h+\frac{m}{2}).$$
Thus, we only need to consider the case when $k\geq m+1.$ It suffices to see that $v^{[k]}(h+\frac{m}{2})$ is a multiple of $v^{[k-m]}.$
Note that roots of $v^{[k-m]}$ are of the form $j+t$, where $t$ is a root of $v$ and $j<k-m.$ Write $t=t'+p/2$ with $h(t')=0$
and odd $p$ with $|p|\leq m.$ Then $$v^{[k]}(t+\frac{m}{2})=v^{[k]}(t'+(p+m)/2)=0,$$ and we are done. 
\end{proof}

\begin{remark}
We assume again that $A(v)$ is simple. It was observed by Hodges \cite{H} that $A(v')$ is Morita equivalent to $A(v)$
if roots of $v'$ are integer translates of roots of $v.$ Next we recall a result by Jordan and Wells \cite{JW} describing the fixed ring
of $A(v)$ under the natural diagonal action of a cyclic group. Namely, $\mathbb{Z}/l\mathbb{Z}$ acts diagonally
on $A(v)$ by $$\sigma(x)=\xi x,\quad \sigma(y)=\xi^{-1}y,\quad \sigma(h)=h,$$ where $\xi$ is a primitive $l$-th root of unity.
Then the corresponding fixed ring $A(v)^{\mathbb{Z}/l\mathbb{Z}}$ is isomomrphic to $A(v'),$
where $v'=\prod_{i=0}^{l-1}v(h+i/l).$ Now, Theorem \ref{typeA} can be reformulated as follows:
A simple ring $A(v)$ is a fixed ring of a domain $R$ under a finite subgroup of automorphisms $G$ if and only if $A(v)$ is Morita equivalent to $A(v')^{\mathbb{Z}/l\mathbb{Z}}$ for some
$v'$ (under the diagonal action of $\mathbb{Z}/l\mathbb{Z}$ on $A(v'))$ and $G\cong \mathbb{Z}/l\mathbb{Z}.$
\end{remark}

\section{The case of quantum tori}

Recall that given an associative algebra, its Picard group $\Pic(A)$ is defined as the group of isomorphism classes of invertible
$A$-bimodules under the tensor product. We have the natural homomorphism $\Out(A)\to \Pic(A),$ 
where $\Out(A)$ denotes the group of outer automorphisms of $A.$



Given $q\in\mathbb{C}^*$, the corresponding quantum torus, to be denoted by $A_q,$ is defined as a $\mathbb{C}$-algebra with generators $x, y$ and their inverses $x^{-1}, y^{-1}$ with the relation $xy=qyx.$ An $n$-dimensional quantum
torus is defines as follows: 
$$A^n_{q}=\bigotimes_{i=1}^nA_q.$$

 It is natural to ask whether  the natural homomorphism  $$\Out(A^n_q)\to\Pic(A_q^n)$$ is an isomorphism. This was proved to be the case
for a quantum torus by Berest, Ramadoss and Tang \cite{BRT}. Their proof
is based on the description of the isomorphism classes of ideals of $A_q$. We generalize this result for $n$-dimensinal quantum tori.
An upshot of our proof (already for $n=1$ case) is that it does not rely on any nontrivial facts about ideals in $A_q^n.$

In what follows, given an automorphism $\phi\in\Aut(A)$, by $A_{\phi}$ we denote $A$ viewed as a $A$-bimodule
with the usual left action and the right action twisted by $\phi.$
The next result provides a criterion for injectivity of the natural restriction homomorphism $\Pic(A)\to\Aut(Z)$, where
$A$ is an algebra finite over its center $Z$. It is (mostly) well-known,
we include the proof for the reader's convenience.

\begin{lemma}\label{Picard}
Let $\bold{k}$ be an algebraically closed field. Let $R$ be a $\bold{k}$-algebra which is finite over its center $Z$, such
that $Z$ is a finitely generated $\bold{k}$-domain. Let $U\subset \spec(Z)$ be the Azumaya locus of $R.$
Assume that $Z$ is normal and the compliment of $U$ in $\spec(Z)$ has codimension $\geq 2$. 
Moreover, assume that $R$ is a Cohen-Macaulay $Z$-module in codimenion 2.
Then the natural restriction homomorphism
$\Aut(R)\to \Aut(Z)$ extends to a homomorphism $\Pic(R)\to \Aut(Z)$, which is injective if the Picard group of $U$ is trivial.

\end{lemma}

\begin{proof}

Let $M$ be an invertible $R$-bimodule. It follows from a standard argument that the support
of $M$ on $U\times U$ must be a graph of an automorphism of
 $U.$ Since the codimension of the compliment of $U$ in $\spec Z$ is at least 2 and $Z$ is normal, we get that $\Aut(U)\leq \Aut(\spec(Z)).$ 
 Thus we obtain the desired homomorphism $\Pic(R)\to \Aut(Z).$

Now, we assume that the Picard group of $U$ is trivial. 
 If $M\in\Pic(R)$ maps to $\Id_Z$, then $M_U$ is supported on the diagonal of $U\times U.$
 Thus, $M$ is a module over $$R_U\otimes_{\mathcal{O}_U}R^{op}_U\cong \End_{\mathcal{O}_U}(R|_U).$$
  Hence, $M_{U}$ must be of the form $R_{U}\otimes_{\mathcal{O}_U}N$ where $N\in \Pic(U).$ Since the Picard group of
  $U$ is trivial, it follows that $M_U\cong R_U.$ Let $I$ denote the defining ideal of the compliment of $U$ in $\spec(R).$
  Since $M$ is a projective left $R$-module,
it follows from our assumption that $depth_IR, depth_IM\geq 2$ (here $R, M$ are viewed as $Z$-modules).
Now a standard argument using local cohomology shows that $\Gamma(U, M_U)=M$ and $\Gamma(U, R_U)=R.$ So, $M\cong  R$ and we are done.

\end{proof}

\begin{theorem}\label{picard}
Let $A^n_q$ be an $n$-dimensional quantum torus with $q$ not a root of unity. Then the natural map $\Out(A^n_q)\to \Pic(A^n_q)$ is an isomorphism.

\end{theorem}
\begin{proof}
We put $A=A_q$ for simplicity.
Let $M$ be an invertible $A$-bimodule. We need to show that $M\cong A$ as a left $A$-module.
Let $S\subset\mathbb{C}$ be a finitely generated ring over which $A, M$ are defined and $M$ is still an invertible
bimodule over $S.$ We show that $M_{\bold{k}}\cong A_{\bold{k}}$ as left modules for all base changes
$S\to\bold{k},$ where $\bold{k}$ is a finite field and
$char(\bold{k})=p\gg 0.$
Let $\bold{k}$ be a finite field, so $\bar{q}$ (the image of $q$ in $\bold{k}$) is an $l$-root of unity, for some $l.$
Then it is well-known that the center of $A_{\bold{k}}$, which we denote by $Z_{\bold{k}},$ is isomophic to the ring of Laurent polynomials:
$$Z_{\bold{k}}=\bold{k}[x_1^{\pm l}, \cdots, x_n^{\pm l}, y_1^{\pm l}, \cdots, y_n^{\pm l}].$$ 
The corresponding Poisson bracket on $Z_{\bold{k}}$ is given as follows:
$$\lbrace x_i^l, y_j^l\rbrace =\delta_{ij}lx_i^ly_j^l,\quad \lbrace x_i^l, x^l_j\rbrace=\lbrace y^l_i, y^l_j\rbrace=0.$$
Since $M_{\bold{k}}$ is an invertible bimodule, it follows from a standard argument that its support on $\spec(Z_{\bold{k}}\otimes_{\bold{k}} Z_{\bold{k}})$ must be a graph of an automorphism
 $\phi\in \Aut(Z_{\bold{k}}).$ Moreover $\phi$ must preserve the Poisson bracket on $Z_{\bold{k}}.$
Next we check that there exists $\psi\in \Aut(A_{\bold{k}})$
such that $\psi|_{Z_{\bold{k}}}=\phi.$ 
Recall that by $P\Aut(Z_{\bold{k}})$ we denote the group of automorphisms of $Z_{\bold{k}}$ preserving the Poisson bracket.
It is well-known and easy to prove that $P\Aut(Z_{\bold{k}})\cong Sp(2n, \mathbb{Z})$. As $Sp(2n, \mathbb{Z})$ also acts
on $A$ by automorphisms, we can now easily conclude that there exists $$\psi\in Sp(2n, \mathbb{Z})\leq \Aut(A_{\bold{k}}),$$ 
such that $\psi|_{Z_{\bold{k}}}=\phi.$ 
Then $(A_{\bold{k}})_{\psi}$ and $M_{\bold{k}}$ have the same support. Now, since the Picard group of $\spec(Z_{\bold{k}})$ is trivial, it follows from 
Lemma \ref{Picard}
that
the restriction homomorphism $$\Pic(A_{\bold{k}})\to\Aut(Z_{\bold{k}})$$ is injective, it follows that $M_{\bold{k}}\cong (A_{\bold{k}})_{\psi}$.
Hence $M\cong A_{\bold{k}}$ as left modules.

Since $M\in \Pic(A)$, we may assume that $M$ is a left ideal in $A$ (as a left module) (see for example [\cite{BEG}, Lemma 3]) . We need to show that it is a principal ideal.
Assume that this is not the case. Then $\Gr(M)$ is not a principal ideal in $\Gr(A)$. Then for any $p\gg 0,$ there exists
a base change $S\to \bold{k}$ with $char(\bold{k})=p$, such that $\Gr(M_{\bold{k}})$ is
not principal in $\Gr(A_{\bold{k}})$, so $M_{\bold{k}}$ is not isomorphic to $A_{\bold{k}}$, a contradiction.
\end{proof}

The next result solves the inverse Galois problem for quantum tori.

\begin{theorem}\label{FixedTorus}
Let $q\in\mathbb{C}^*$ be a non-root of unity. Let $A_q=R^G$, where $R$ is a $\mathbb{C}$-domain and $G$ is a finite subgroup of $\mathbb{C}$-automorphisms of
$R.$ Then $R\cong A_{q'}$ for some $q'\in\mathbb{C}^*.$
\end{theorem}

Just as for the analogous result for spherical subalgebras of rational Cherednik algebras, the proof of Theorem \ref{FixedTorus}  crucially utilizes
 Harish-Chandra bimodules over $A_q$. We recall their definition next.
In what follows, given $a\in A^*$ and an $A$-bimodule $M,$ we denote by $\Ad(a)\in \End(M)$ the conjugation by $a.$

\begin{defin}

Let $M$ be a $A_q$-bimodule. Then $M$ is a Harish-Chandra bimodule if $Ad(x), Ad(y)$ act semi-simply on $M.$

\end{defin}

We also need the following result from algebraic number theory, which is a special case of theorem of A. Perucca \cite{P}.

\begin{lemma}\label{Perruca}[\cite{EW}, Corollary A.2]
Let $S\subset\mathbb{C}$ be a finitely generated ring. Let $0\neq q\in S, d\in\mathbb{N}$. Then there exist
infinitely many base changes $\chi:S\to \bold{k}$, where $\bold{k}$ is a finite field and
$\chi(q)$ is a root of unity of order coprime to $d.$

\end{lemma}

\begin{proof}[Proof of Theorem \ref{FixedTorus}]
Once again, we put $A_q=A.$
The proof goes along the lines of the proof of Theorem \ref{HC}.
We may assume that $G$ is a simple group. 
The first step is to show that $R$ is a Harish-Chandra $A$-bimodule.
Let $v\in A$, since $R$ is a finite left $A$-module, we may write
$$v^n=\sum_{i<n} h_iv^i, \quad h_i\in A.$$
Let $m$ be the largest degree of $y, y^{-1}$ in $h_i.$ Let $k=|G|!.$
Then we show that $$\prod_{|i|\leq mk}(\Ad(x)-q^{i/k})v=0,$$
which implies that $v$ is a sum of eigenvectors of $\Ad(x).$ Repeating the same argument for $\Ad(y)$ gives that $R$ is a Harish-Chandra bimodule.

Let $S\subset \mathbb{C}$ be a large enough
finitely generated ring over which everything is defined containing $|G|$-th primitive roots of unity and $q^{1/k}\in S, k\leq m|G|.$
Localizing $S$ if necessary and using Lemma \ref{Perruca}, it suffices to show the above equality holds in any base change $S\to\bold{k},$ where $\bold{k}$ is a finite field of large enough characteristic,
such that $\bar{q}$ (the image of $q$ in $\bold{k}$) is a root of unity of order coprime to $|G|.$

Indeed, we know that (just as in the proof of Theorem \ref{HC}) given a base change $S\to\bold{k}$ such that $\bar{q}$ (the image of $q$ in $\bold{k}$) is a root of unity,
then $A_{\bold{k}}$ is an Azumaya algebra, the restriction map $\spec(Z(R_{\bold{k}}))\to \spec (Z(A_{\bold{k}}))$ is a $G$-Galois etale covering, and  
$$R_{\bold{k}}\cong A_{\bold{k}}\otimes_{Z(A_{\bold{k}})}Z(R_{\bold{k}}).$$ 
But 
$\spec Z(A_{\bold{k}})\cong \mathbb{A}^*_{\bold{k}}\times \mathbb{A}^*_{\bold{k}},$
and it is well-known that any connected $p'$-etale covering of $\mathbb{A}^*_{\bold{k}}\times \mathbb{A}^*_{\bold{k}}$
must be isomorphic to $\mathbb{A}^*_{\bold{k}}\times \mathbb{A}^*_{\bold{k}}.$
 Therefore, we can conclude that $$\spec(Z(R_{\bold{k}}))=\bigsqcup_i X_i,$$
  where $X_i\cong \mathbb{A}^*_{\bold{k}}\times \mathbb{A}^*_{\bold{k}}$
and the etale coverings $X_i\to \mathbb{A}^*_{\bold{k}}\times \mathbb{A}^*_{\bold{k}}$ are given  in terms of coordinates by $(z, w)\to (z^{a_i}w^{b_i}, w^{c_i}).$
Let $A_i=\begin{pmatrix}
a_i & b_i &\\
0 & c_i
\end{pmatrix}$
be viewed as homomorphisms $A_i:\mathbb{Z}^2\to\mathbb{Z}^2.$ Then $\mathbb{Z}^2/Im(A_i)$ is isomorphic to the etale Galois group of the above covering.
In particular, $\det(A_i)$ must divide $|G|.$  Also, without loss generality, absolute values of all entries of $A_i$ are bounded by $|G|.$
 
  So, we have 
$$Z(R_{\bold{k}})\cong \prod_ie_i\bold{k}[z_i, w_i, z_i^{-1}, w_i^{-1}],$$ where 
$e_i$ are pairwise orthogonal idempotents $\sum_i e_i=1,$ and 
$$z_i^{a_i}w_i^{b_i}=x_i^l, \quad w_i^{c_i}=y_i^l,$$
where $x_i=e_ix, y_i=e_iy,$
with $a_ic_i$ dividing $|G|.$ Next, we show that $R_{\bold{k}}$ must be a product of quantum tori, and hence cannot contain any nontrivial nilpotent element.
Write $$1=t_ia_ic_i+t_i'l, t_i, t'_i\in\mathbb{Z}.$$
Now put $$f_i=x_i^{tc_i}y_i^{-tb_i}z_i^{t'_i},\quad g_i=y_i^{t_ia_i}w_i^{t_i'}.$$
 Then 
$$z_i=f_i^l,\quad w_i=g_i^l.$$ Now it follows that $f_i, g_i$ generate $e_iR_{\bold{k}}$, so
$$R_{\bold{k}}\cong \prod_ie_i\bold{k}\langle f_i, g_i\rangle.$$
In particular, $R_{\bold{k}}$ has no nonzero nilpotent element.
We have that $\Ad(x)(g_i)=\bar{q}^{1/c_i}.$
Next, we claim that the degree of $g_i, g_i^{-1}$ in $e_i\bar{v}$ is at most $n|G|.$
Indeed, we have $e_i\bar{v}^n=\sum e_ia_i\bar{v}^i$ in $\bold{k}\langle f_i, g_i\rangle,$
and the degree of $e_ia_i$ in $x, x^{-1}$ is at most $n,$ hence its degree is at most $n|G|$ in $f_i, f_i^{-1}$
which implies the desired result. Now it follows that for all $i$, we have
$$\prod_{j<n|G|}(\Ad(x)-\bar{q}^{j/c_i})(e_i\bar{v})=0.$$
Since $\sum_ie_iv=v$ and $c_i||G|,$ we get that 
$$\prod_{j\leq n|G|^2}(\Ad(x)-\bar{q}^{j/|G|})(\bar{v})=0.$$
This completes the proof that $R$ a Harish-Chandra bimodule over $A.$

Next, we claim that $R_{\bold{k}}$ has no nontrivial idempotents. Indeed, let $e\in R_{\bold{k}}$ be an idempotent.
Since $R_{\bold{k}}$ has no nonzero nilpotent elements then
given any $\mathbb{Z}$-grading of $R_{\bold{k}}$, then $e$ must have degree 0.
Let $X\subset \mathbb{C}^*$ denote the subgroup of weights of $R$ with respect to $Ad(x).$ 
Clearly $X$ is finitely generated.
Then any homomorphism $\chi:X\to \mathbb{Z}$
induces a $\mathbb{Z}$-grading on $R$. Thus, by varying  characters $\chi\in Hom(X, \mathbb{Z})$ 
and considering the corresponding $\mathbb{Z}$-gradings on $R_{\bold{k}}$ (so $\deg R^{\alpha}=\chi(\alpha)$) (and replacing $\Ad(x)$ by $\Ad(y)$) we may 
conclude that $e$ must belong to
$R'_{\bold{k}}$, where $R'$ is a subring of $R$ spanned by eigenvectors of $Ad(x), Ad(y)$ whose eigenvalues are roots of unity.
Denote by $R^0$ the centralizer of $x$ and $y$ in $R.$
 Clearly $R^0$ is $\Gamma$-invariant and $R'\cap A=R'^{\Gamma}=Z(A)=S.$ Hence $R^0$ is a finite $S$-algebra. Since $R$ is an integral domain, it follows
 that $R^0=S.$ 
Let $z\in R'$ be a common eigenvector of $Ad(x), Ad(y).$ So, $z^m$ commutes with $x, y$ for some $m\geq 1$. Hence $z^m\in S,$ which implies that
$z\in S$. Since $R'$ is spanned over $S$ with such elements, we conclude that  
 $R'=S.$
It follows that $e \in\bold{k}.$ Thus $G$ must be a quotient
of the etale fundamental group of $\spec(Z(A_{\bold{k}}))\cong \mathbb{A}^{*}_{\bold{k}}\times\mathbb{A}^*_{\bold{k}}$, so $G=\langle \sigma\rangle$ a cyclic group of a prime order $l.$

Next, we remark that any proper $G$-invariant subring $B\subset R$ strictly containing $A$ must equal to $R.$
Indeed, since $B^{G}=A$, we may argue just as we did above for $R$ and conclude that $Z(B_{\bold{k}})$ is a subring of $Z(R_{\bold{k}}), B_{\bold{k}}\cong A_{\bar{q}}\otimes_{Z(A_{\bar{q}})}Z(R_{\bold{k}})$ 
and $\spec(Z(B_{\bold{k}}))\to \spec Z(A_{\bar{q}})$
is a $G$-Galois etale covering. Thus, $Z(B_{\bold{k}})=Z(R_{\bold{k}})$ and hence $B_{\bold{k}}=R_{\bold{k}}$. So, $B=R$ as desired.

Let $\xi$ be a primitive $l$-th root of unity. Let $z\in R$ be a common eigenvector of $Ad(x), Ad(y),$ with eigenvalues $\alpha, \beta,$ respectively, such that $\sigma(z)=\xi z.$ 
Then $$A\neq \bigoplus_{i=0}^{l-1}Az^i$$ is a $\Gamma$-invariant subring of $R,$ thus it must equal $R.$ We have that $z^l=x^ny^m$
for some $n, m\in\mathbb{Z}.$
 As $x$ (or $y$) can be replaced by any $x^ay^b$ with $(a, b)=1,$
we may assume without loss of generality that $z^l=x^n.$ 
So, $xzx^{-1}=\eta z,$ where  $\eta^l=1.$
Also, $(l, n)=1$ as $R$ is a domain.
Let $an=bl+1$ for $a, b\in\mathbb{Z}.$ Put $t=z^ax^{-b}.$ Then $t^l=cx, t^n=dz$ for some $c, d\in\mathbb{C}^*.$
Therefore, $R$ is generated by $t, y$ and  $yty^{-1}=q't$ for some  $q'\in\mathbb{C}^*.$ Hence $R$ is a quantum torus, as desired.

\end{proof}

\begin{acknowledgement}
I am very grateful to the anonymous referee for making numerous helpful suggestions.
\end{acknowledgement}




\begin{thebibliography}{}

\bibitem[AHV]{AHV}
J. Alev, T.J. Hodges, J.-D. Velez, {\em Fixed rings of the Weyl algebra $A_1(\mathbb{C})$}, J. Algebra 130 (1) (1990)
83--96.

\bibitem[AP]{AP}
J.~Alev, P.~Polo, {\em A rigidity theorem for finite group actions on enveloping algebras of semi-simple Lie algebras},
Advances in Math.  111(1995) no.2 208--226.

\bibitem[BEG]{BEG}
Y.~Berest, P.~Etingof, V.~Ginzburg, {Morita equivalence for Cherednik algebras}, J. Reine Angew. Math. 568 (2004), 81--98.

\bibitem[BFG]{BFG}
R.~Bezrukavnikov, M.~Finkelberg, V.~Ginzburg, {\em Cherednik algebras and Hilbert schemes in characteristic p},
Represent. Theory 10 (2006), 254--298. 

\bibitem[BK]{BK}
R.~Bezrukavnikov, D.~Kaledin, {\em Fedosov quantization in positive characteristic},
J. Amer. Math. Soc. 21 (2008), no. 2, 409–438. 


\bibitem[BKK]{BKK}
A.~Belov-Kanel, M~Kontsevich, {\em The Jacobian Conjecture is stably equivalent
to the Dixmier Conjecture}, Moscow Mathematical Journal, 7 (2007), no.2, 209--218.


\bibitem[BRT]{BRT}
Y.~Berest, A.~Ramadoss, X.~Tang, {\em The Picard group of a noncommutative algebraic torus},
J. Noncommut. Geom. 7 (2013), no. 2, 335--356.
\bibitem[C]{C}
P.~Caldero, {\em Isomorphisms of finite invariants for enveloping algebras, semi-simple case}, Advances in Mathematics, Vol 134, No 2, (1998), 294--307.


\bibitem[EG]{EG}
P.~Etingof, E.~Ginzburg, {\em  Symplectic reflection algebras, Calogero-Moser space, and deformed Harish-Chandra homomorphism},
Invent. Math. 147 (2002), no. 2, 243--348. 


\bibitem[EW]{EW}

P.~Etingof, C.~Walton, {\em Finite dimensional Hopf actions on algebraic quantizations},
Algebra Number Theory 10 (2016), no. 10, 2287--2310.





\bibitem[GW]{GW}
J.~Gaddis, R.~Won, {\em  Fixed rings of generalized Weyl algebras},
J. Algebra 536 (2019), 149--169 


\bibitem[H]{H}

T~.Hodges, {\em  Noncommutative deformations of type-$A$ Kleinian singularities},
J. Algebra 161 (1993), no. 2, 271--290.           


\bibitem[JW]{JW}

D.~Jordan, I.~Wells. {\em Invariants for automorphisms of certain iterated skew
polynomial rings},  Proc. Edinburgh Math. Soc. (2), 39(3):461--472, 1996.

\bibitem [KR]{KR}

V.~Kac, A.~Radul, {\em Poisson structures for restricted Lie algebras}, The Gelfand Mathematical Seminars, 1996--1999.

\bibitem[LO]{LO}
M.~Lieblich, M.~Olsson, {\em Generators and relations for the etale fundamental group}, Pure Appl. Math. Q. 6 (1) (2010) 209--243.


\bibitem [M]{M}
S.~Montgomery, {\em Fixed Rings of Finite Automorphism Groups of Associative Rings}, Lecture Notes in Math., 1980.



\bibitem[N]{N}
 Y. ~Namikawa, {\em Fundamental groups of symplectic singularities}, Higher dimensional algebraic geometryin honour of Professor Yujiro Kawamata’s sixtieth birthday, 321334, Adv. Stud. Pure Math., 74,
Math. Soc. Japan, Tokyo, 2017.

\bibitem[P]{P}
A.~Perucca, {\em Prescribing valuations of the order of a point in the reductions of abelian varieties and tori}, Journal of Number Theory
Volume 129, Issue 2, February 2009, Pages 469-476

\bibitem [R]{R}
L.~Rybnikov, {\em Argument shift method and Gaudin model}, Functional Analysis and Its Applications (2006), 188--199.



\bibitem[S]{S}
S.~P. Smith, {\em Overrings of primitive factor rings of $U({sl}(2,{\bf C}))$},
J. Pure Appl. Algebra 63 (1990), no. 2, 207--218.

\bibitem[St]{St}
J. T. Stafford, {\em Endomorphisms of right ideals of the Weyl algebra},
Trans. Amer. Math. Soc. 299 (1987), no. 2, 623--639.



\bibitem [T1]{T1}
A.~Tikaradze, {\em On the Azumaya locus of almost commutative algebras,}
Proc. Amer. Math. Soc. 139 (2011), no. 6, 1955–1960.

\bibitem[T2]{T2}

A.~Tikaradze, {\em On automorphisms of enveloping algebras,} Int. Math. Res. Not. IMRN 2020, no. 21, 8183--8196.

\bibitem[T3]{T3}

A.~Tikaradze, {The Weyl algebra as a fixed ring}, Adv. Math. 345 (2019), 756--766.

\bibitem[T4]{T4}
A.~Tikaradze, {Derived invariants of the fixed ring of enveloping algebras of semisimple Lie algebras}, Math. Z. 297 (2021), no. 1-2, 475--481.

\bibitem[Ts]{Ts}
Y.~Tsuchimoto,
{\em Endomorphisms of Weyl algebra and $p$-curvatures},
Osaka J. Math. 42 (2005), no. 2, 435--452.







\bibitem[VWW]{VWW}
Vu, Van H,  M.~Wood, Melanie, P.~Wood,
{\em Mapping incidences},
J. Lond. Math. Soc. (2) 84 (2011), no. 2, 433--445.

\bibitem[W]{W}
G.~Wilson, {\em  Collisions of Calogero-Moser particles and an adelic Grassmannian},
Inv. Math. 133 (1998), 1--41.    


\end{thebibliography}
\end{document}